\documentclass[fleqn]{article}
\usepackage{amssymb}
\usepackage{amsthm}
\usepackage{amsmath}
\usepackage{fontenc}
\usepackage{inputenc}
\usepackage{enumitem}
\usepackage{hyperref}
\usepackage{xcolor}
\usepackage{tikz}
\usetikzlibrary{scopes,arrows,decorations.pathmorphing,backgrounds,positioning,fit,petri}

\def\MR#1{\href{http://www.ams.org/mathscinet-getitem?mr=#1}{MR#1}}
\def\real{\hbox{\rm\setbox1=\hbox{I}\copy1\kern-.45\wd1 R}}
\def\natural{\hbox{\rm\setbox1=\hbox{I}\copy1\kern-.45\wd1 N}}

\newtheorem{theorem}{Theorem}[section] % 1st argument is your name for it
\newtheorem{lemma}[theorem]{Lemma}     % 2nd argument is what is printed
\newtheorem{corollary}[theorem]{Corollary}

\newtheorem{case}[theorem]{Case}
\newtheorem{property}[theorem]{Property}
\title{Tower multiplexing and slow weak mixing} % This is the full title of the paper
\author{Terrence M. Adams
\thanks{
Terrence M. Adams is with the Department of Defense, 
9161 Sterling Drive 
Laurel, MD 20723
United States of America}}

\date{\today}

\begin{document}
\maketitle

\begin{abstract}
A technique is presented for multiplexing two ergodic measure preserving transformations together to derive a third limiting transformation.  This technique is used to settle a question regarding rigidity sequences of weak mixing transformations.  Namely, given any rigidity sequence for an ergodic measure preserving transformation, there exists a weak mixing transformation which is rigid along the same sequence. 
This establishes a wide range of rigidity sequences 
for weakly mixing dynamical systems. 
\end{abstract}
\section{Introduction}
Fix a Lebesgue probability space. 
Endow the set of invertible measure preserving transformations with the weak topology. 
It is well known that both the properties of weak mixing and rigidity are generic properties in this topological space \cite{Hal56}. 
This is interesting since the behaviors of these two properties contrast greatly. Weak mixing occurs when a system 
equitably spreads mass throughout the probability space for most times. Rigidity occurs when a system 
evolves to resemble the identity map infinitely often. Since both of these behaviors exist simultaneously 
in a large class of transformations, it is natural to ask what types of rigidity sequences are realizable 
by weak mixing transformations. Here we resolve this question to show all rigidity sequences are realizable 
by the class of weak mixing transformations. 
\begin{theorem}
\label{towerplex0}
Given any ergodic measure preserving transformation $R$ on a Lebesgue 
probability space, and any rigid sequence $\rho_n$ for $R$, there exists 
a weak mixing transformation $T$ on a Lebesgue probability space such that 
$T$ is rigid on $\rho_n$. 
\end{theorem}
Prior to proving this main result, we present a new and direct method for combining two invertible ergodic finite measure preserving transformations to obtain a third limiting transformation.  The technique iteratively utilizes the Kakutani-Rokhlin lemma (\cite{Kak43},\cite{Rok48}). A measure preserving transformation $T$ on a separable probability space $(X,\mathbb{B},\mu)$ is ergodic if any invariant measurable set $A$ has measure 0 or 1. In particular, $TA=A$ implies $\mu (A)=\mbox{0 or 1}$.
\begin{lemma}{(Kakutani 1943, Rokhlin 1948)}
\label{kr-lemma}
Let $T:X\to X$ be an ergodic measure preserving transformation on a nonatomic probability space $(X, \mathcal{B}, \mu )$, $h$ a positive integer and $\epsilon >0$. There exists a measurable set $B\subset X$ such that 
$B,TB,\ldots ,T^{h-1}B$ are pairwise disjoint and 
$\mu (\bigcup_{i=0}^{h-1}T^i(B))>1-\epsilon$. The collection $\{B, TB,\ldots ,T^{h-1}B\}$ is referred to as a Rokhlin tower 
of height $h$ for transformation $T$. 
\end{lemma}
Clearly, this lemma demonstrates that any ergodic measure preserving 
transformation can be approximated arbitrarily well by periodic 
transformations in an appropriate topology (i.e. uniform topology). 
See Halmos \cite{Hal44}, \cite{Hal56}, Rokhlin \cite{Rok49}, 
Katok and Stepin \cite{KS67}. 
Much of the early work in this regard focuses on the 
topological genericity of specific properties of measure 
perserving transformations. 
In \cite{KS67}, results are presented on rates of approximation by periodic 
transformations, and connections with dynamical properties. 
Recent research of Kalikow demonstrates the utility of developing a general 
theory of Rokhlin towers \cite{Kal12}. 
Also, it is clear from the Kakutani-Rokhlin lemma that 
any ergodic measure preserving tranformation can be approximated 
arbitrarily well by another ergodic measure preserving transformation 
from any isomorphism class. This observation is utilized repeatedly 
in this work. 

Two input transformations $R$ and $S$ are multiplexed together to derive 
an output transformation $T$ with prescribed properties.  The multiplexing operation 
is defined using an infinite chain of measure theoretic isomorphisms. 
In the case where $R$ is ergodic and rigid, and $S$ weak mixing, we present 
a method for unbalanced multiplexing of $R$ and $S$.  
Over time, transformations isomorphic to $R$ are used 
on a higher proportion of the measure space, as the action 
by $S$ dissipates over time. 
We refer to this process informally as slow weak mixing. 

A measure preserving transformation $T:X\to X$ 
is weak mixing, if for all measurable sets $A$ and $B$, 
$$
\lim_{n\to \infty} \frac1n \sum_{i=0}^{n-1}\vert 
\mu (T^iA\cap B) - \mu (A)\mu (B)\vert =0.
$$
Clearly, if $T$ is weak mixing, then $T$ is ergodic. 
Also, $T$ is weak mixing, if and only if $T$ has only 1 
as an eigenvalue, and all eigenfunctions are constants 
almost everywhere.
%\footnote{For the remainder of this paper, 
%statements of equality are assumed to hold almost everywhere.}
An ergodic measure preserving transformation $R$ is rigid 
on a sequence $\rho_n\to \infty$, if for any measurable set $A$, 
$$\lim_{n\to \infty}\mu (T^{\rho_n}A\triangle A)=0.$$
The sequence $\rho_n$ is called a rigidity sequence for $R$. 

Several forms of rigidity have been studied in both ergodic theory 
and topological dynamics. In the case of topological dynamics, 
both rigidity and uniform rigidity are considered. Uniform rigidity 
was introduced in \cite{GlaMao89} and given a specific generic characterization. 
In \cite{Jam09}, it is shown the notion of uniform rigidity is mutually 
exclusive from measurable weak mixing on a Cantor set. 
In particular, every finite measure preserving weak mixing 
transformation has a representation that is not uniformly rigid. 
Weak mixing and rigidity have been studied for interval exchange transformations. 
See \cite{Cha12} and \cite{AviFor04} for recent results in this regard. 
Rigid, weak mixing transformations have been studied in the setting of 
infinite measure preserving transformations, as well as nonsingular transformations. 
Mildly mixing transformations are finite measure preserving transformations 
that do not contain a rigid factor. These are the transformations which yield 
ergodic products with any infinite ergodic transformation \cite{FurWei78}. 
See works \cite{Aar79}, \cite{AarHosLem12}, \cite{BdJLR} and the references 
therein for results related to notions of weak mixing and rigidity 
for infinite measure preserving or nonsingular transformations. 
The notion of IP-sequences was introduced 
by Furstenberg and Weiss in connection with rigid transformations. 
There has been recent research on IP-rigidity sequences (i.e. IP-sequences which 
form a rigid sequence) for weak mixing transformations. 
See \cite{BdJLR} and \cite{Gri12} for results on IP-rigidity. 

The notion of rigidity was extended to $\alpha$-rigidity by Friedman \cite{Fri89}. 
Transformations are constructed which are 
$\alpha$-rigid and $(1-\alpha)$-partial mixing for any $0<\alpha < 1$. 
See \cite{FriGabKin88} and \cite{Abd09} for further research on 
$\alpha$-rigid transformations. Many of these notions have been studied 
for more general group actions. See \cite{BerGor05} for a survey of 
weak mixing group actions.  Since our results depend mainly on the use 
of Lemma \ref{kr-lemma} which extends to more general groups 
(i.e. amenable, abelian), there should exist an extension of techniques 
provided in this work to a wider class of groups. 
Since some of the principles provided in this work appear new, we 
focus exclusively on the case of measure preserving 
$\mathbb{Z}$-actions on $[0,1)$ with Lebesgue measure. 

For a recent comprehensive account on rigidity sequences, 
we recommend recent publications \cite{BdJLR} and \cite{EisGri}.
Both of these works provide much detail on the current understanding of rigidity 
for weak mixing transformations. 

%Observe that any discrete spectrum transformation admits rigidity sequences 
%whose density approaches zero arbitrarily slow. Given a fixed discrete spectrum 
%transformation and a fixed weak mixing transformation, then there will always 
%exist a zero density sequence which is rigid for the discrete spectrum transformation, 
%but mixing for the weak mixing transformation. However, the results in this paper 
%demonstrate that if the rigidity sequence is fixed, then there exists a weak mixing 
%transformation with the same rigidity sequence. 

\section{Towerplex Constructions}
The main result is established constructively using Lemma \ref{kr-lemma}. Given two transformations $R$ and $S$, we define a third transformation $T$ which is constructed as a blend of $R$ and $S$, such that $T$ acts more like $R$, asymptotically. We will define a sequence of positive integers $h_n, n\in \natural$, and a sequence of real numbers $\epsilon_n >0$ such that 
$\sum_{n=1}^{\infty}{1}/{h_n} < \infty$ and $\sum_{n=1}^{\infty}\epsilon_n < \infty$. Also, let $r_n$ and $s_n$ 
for $n\in \natural$ be sequences of real numbers satisfying: 
$0\leq r_n, s_n\leq 1$. 
\subsection{Initialization}
Suppose $R$ and $S$ are ergodic measure preserving transformations defined 
on a Lebesgue probability space $(X,\mu,\mathbb{B})$. Partition $X$ into two equal 
sets $X_1$ and $Y_1$ (i.e. $\mu (X_1)=\mu (Y_1)=1/2$). 
Initialize $R_1$ isomorphic to $R$ and $S_1$ isomorphic to $S$ to operate on $X_1$ and $Y_1$, respectively. Define $T_1(x)=R_1(x)$ for $x\in X_1$ and $T_1(x)=S_1(x)$ for $x\in Y_1$.  Produce Rohklin towers of height $h_1$ with residual less than $\epsilon_1/2$ for each of $R_1$ and $S_1$. In particular, let $I_1, J_1$ be the base of the $R_1$-tower and $S_1$-tower such that $\mu (\bigcup_{i=0}^{h_1-1}R_1^iI_1)>1/2(1-\epsilon_1)$ and 
$\mu (\bigcup_{i=0}^{h_1-1}S_1^iJ_1)>1/2(1-\epsilon_1)$.  
Let 
$X_1^*=X_1\setminus \bigcup_{i=0}^{h_1-1}R_1^i(I_1)$ 
and 
$Y_1^*=Y_1\setminus \bigcup_{i=0}^{h_1-1}S_1^i(J_1)$ be the 
residuals for the $R_1$ and $S_1$ towers, respectively. 
Choose $I^{\prime}_1\subset I_1$ and $J^{\prime}_1\subset J_1$ such that 
$$
\mu (I^{\prime}_1)=r_1\mu (I_1)\mbox{ and }\mu (J^{\prime}_1)=s_1\mu (J_1).
$$
%Define an invertible map $\phi_1:I^{\prime}_1 \to J^{\prime}_1$ such that 
%$\phi_1$ is measure preserving between subspaces 
%$(I^{\prime}_1,\mathbb{B}\cap I^{\prime}_1,\frac{\mu}{\mu (I^{\prime}_1)})$ and 
%$(J^{\prime}_1,\mathbb{B}\cap J^{\prime}_1,\frac{\mu}{\mu (J^{\prime}_1)})$. 
%Extend $\phi_1(x)=S_1^{i}\circ \phi_1\circ R_1^{-i}(x)$ for $x\in R_1^{i}I^{\prime}_1$ and $0\leq i<h_1$. 
%$\phi_1$ is a switching map between subcolumns 
%$\{I^{\prime}_1,R_1I^{\prime}_1,\ldots ,R_1^{h_1-1}I^{\prime}_1\}$
% and 
%$\{J^{\prime}_1,S_1J^{\prime}_1,\ldots ,S_1^{h_1-1}J^{\prime}_1\}$. 
Set $X_2=X_1\setminus [\bigcup_{i=0}^{h_1-1}R_1^i(I_1^{\prime})]\cup [\bigcup_{i=0}^{h_1-1}S_1^i(J_1^{\prime})]$ and 
$Y_2=Y_1\setminus [\bigcup_{i=0}^{h_1-1}S_1^i(J_1^{\prime})]\cup [\bigcup_{i=0}^{h_1-1}R_1^i(I_1^{\prime})]$. 
We will define second stage transformations $R_2:X_2\to X_2$ and $S_2:Y_2\to Y_2$. 
First, it may be necessary to add or subtract measure from
 the residuals so that $X_2$ is scaled properly to define $R_2$, and 
 $Y_2$ is scaled properly to define $S_2$. 
% Let 
% $X_1^*=X_1\setminus \bigcup_{i=0}^{h_1-1}R_1^i(I_1\setminus I^{\prime}_1)$ 
% and 
% $Y_1^*=Y_1\setminus \bigcup_{i=0}^{h_1-1}S_1^i(J_1\setminus J^{\prime}_1)$. 

\subsection{Tower Rescaling}
In the case where $\mu (I_1^{\prime})\neq \mu (J_1^{\prime})$, we give 
a procedure for transferring measure between the towers and the 
residuals. This is done in order to consistently define $R_2$ and $S_2$ 
on the new inflated or deflated towers. 
Let $a=\mu (\bigcup_{i=0}^{h_1-1}R_1^iI_1)$ and 
$b=h_1(\mu (J^{\prime}_1)-\mu (I^{\prime}_1))$. 
Let $c$ be the scaling factor and $d$ the amount of measure 
transferred between $\bigcup_{i=0}^{h_1-1}S_1^i(J^{\prime}_1)$ 
and $X_1^*$. Thus, $a+b-d=ca$ and $1/2-a+d=c(1/2 - a)$. The goal is 
to solve two unknowns $d$ and $c$ in terms of the other values. 
Hence, $d=(1-2a)b$ and $c=1+2b$. 

\subsubsection{$R$ Rescaling}
If $d>0$, define $I_1^*\subset J^{\prime}_1$ such that 
$\mu (I_1^*)=d/h_1$. Let 
$X_1^{\prime}=X_1^*\cup (\bigcup_{i=0}^{h_1-1}R_1^i(I_1^*))$. 
If $d=0$, set $X_1^{\prime}=X_1^*$.  
If $d<0$, transfer measure from $X_1^*$ to the tower.  
Choose disjoint sets $I_1^*(0),I_1^*(1),\ldots ,I_1^*(h_1-1)$ 
contained in $X_1^*$ such that $\mu (I_1^*(i))=d/h_1$. 
Denote $I_1^*=I_1^*(0)$. 
Begin by defining $\mu$ measure preserving map $\alpha_1$ such 
that $I_1^*(i+1)=\alpha_1 (I_1^*(i))$ for $i=0,1,\ldots ,h_1-2$. 
In this case, let 
$X_1^{\prime}=X_1^*\setminus [\bigcup_{i=0}^{h_1-1}I_1^*(i)]$. 

\subsubsection{$S$ Rescaling}
The direction mass is transferred depends on the sign of $b$ above. 
If $d>0$, then $\mu (J^{\prime}_1) > \mu (I^{\prime}_1)$ and 
mass is transferred from the residual $Y_1^*$ to the 
$S_1$-tower. 
Choose disjoint sets $J_1^*(0),J_1^*(1),\ldots ,J_1^*(h_1-1)$ 
contained in $Y_1^*$ such that $\mu (J_1^*(i))=d/h_1$. 
Denote $J_1^*=J_1^*(0)$. 
Begin by defining $\mu$ measure preserving map $\beta_1$ such 
that $J_1^*(i+1)=\beta_1 (J_1^*(i))$ for $i=0,1,\ldots ,h_1-2$. 
In this case, let 
$Y_1^{\prime}=Y_1^*\setminus [\bigcup_{i=0}^{h_1-1}J_1^*(i)]$. 
If $d=0$, set $Y_1^{\prime}=Y_1^*$. 
If $d<0$, transfer measure from the $S_1$-tower to the residual $Y_1^*$. 
Define $J_1^*\subset J_1\setminus J^{\prime}_1$ such that 
$\mu (J_1^*)=d/h_1$. Let 
$Y_1^{\prime}=Y_1^*\cup (\bigcup_{i=0}^{h_1-1}S_1^i(J_1^*))$. 

%The discriminant value $d$ determines in which direction mass is transferred 
%between the tower and its associated residual, and also determines the 
%amount of mass to be transferred. This is necessary to iteratively 
%continue invoking transformations isomorphic to $R$. 
%Note if $d>0$ and mass is transferred between the $R_1$-tower and its residual, 
%then mass will be transferred between the residual $Y_1^*$ and the 
%$S_1$-tower. 

Note, if $d\neq 0$, 
then both $\epsilon_1$ and $\mu (X_1^*)$ may be chosen small enough 
(relative to $r_1$) to ensure the following solutions lead 
to well-defined sets and mappings. For subsequent stages, assume 
$\epsilon_n$ is chosen small enough to force well-defined 
rescaling parameters, transfer sets and mappings $R_n$, $S_n$. 

\subsection{Stage 2 Construction}
We have specified three cases: $d>0$, $d=0$ and $d<0$. The case $d=0$, can be handled along 
with the case $d>0$. This gives two essential cases. Note the case $d<0$ is analogous 
to the case $d>0$, with the roles of $R_1$ and $S_1$ reversed. However, 
due to a key distinction in the handling of the $R$-rescaling and the $S$-rescaling, 
it is important to clearly define $R_2$ and $S_2$ in both cases. 

\begin{case}[$d\geq 0$]
Define $\tau_1:X_1^{\prime}\to X_1^*$ as a measure preserving map between 
normalized spaces 
$(X_1^{\prime},\mathbb{B} \cap X_1^{\prime},\frac{\mu}{\mu (X_1^{\prime})})$ and 
$(X_1^*,\mathbb{B} \cap X_1^*,\frac{\mu}{\mu (X_1^*)})$. 
Extend $\tau_1$ to the new tower base, 
$$\tau_1:[I_1\setminus I_1^{\prime}]\cup [J^{\prime}_1\setminus I_1^*]\to I_1$$
such that $\tau_1$ preserves normalized measure between 
$$\frac{\mu}{\mu ([I_1\setminus I_1^{\prime}]\cup [J^{\prime}_1\setminus I_1^*])}\mbox{ and }\frac{\mu}{\mu (I_1)}.$$
Define $\tau_1$ on the remainder of the tower consistently 
% $$\tau_1:R_1^{h_1-1}(I_1\setminus [I_1^{\prime}\cup I_1^*])\cup 
% S_1^{h_1-1}J^{\prime}_1\to R^{h_1-1}(I_1)$$
such that 
\begin{eqnarray*} 
\tau_1(x)= 
\left\{\begin{array}{ll}
R_1^{i}\circ \tau_1 \circ R_1^{-i}(x) & \mbox{if $x\in R_1^i(I_1\setminus I_1^{\prime})$ for $0\leq i<h_1$} \\ 
R_1^{i}\circ \tau_1 \circ S_1^{-i}(x) & \mbox{if $x\in S_1^{i}(J^{\prime}_1\setminus I_1^*)$ for $0\leq i<h_1$}
\end{array}
\right.
\end{eqnarray*}
Define 
$R_2:X_2\to X_2$ as $R_2=\tau_1^{-1}\circ R_1\circ \tau_1$. Note 
\begin{eqnarray*} 
R_2(x)= 
\left\{\begin{array}{ll}
S_1(x) & \mbox{if $x\in S_1^{i}(J^{\prime}_1\setminus I_1^*)$ for $0\leq i<h_1-1$} \\ 
R_1(x) & \mbox{if $x\in R_1^{i}(I_1\setminus I^{\prime}_1)$ for $0\leq i<h_1-1$} 
\end{array}
\right.
\end{eqnarray*}
Clearly, $R_2$ is isomorphic to $R_1$ and $R$. 

Define $\psi_1:Y_1^{\prime}\to Y_1^*$ as a measure preserving map between 
normalized spaces 
$(Y_1^{\prime},\mathbb{B} \cap Y_1^{\prime},\frac{\mu}{\mu (Y_1^{\prime})})$ and 
$(Y_1^*,\mathbb{B} \cap Y_1^*,\frac{\mu}{\mu (Y_1^*)})$. 
Extend $\psi_1$ to the new tower base, 
$$\psi_1:[J_1\setminus J_1^{\prime}]\cup J_1^*\cup I^{\prime}_1\to J_1$$
such that $\psi_1$ preserves normalized measure between 
$$\frac{\mu}{\mu ([J_1\setminus J_1^{\prime}]\cup J_1^*\cup I^{\prime}_1)}\mbox{ and }\frac{\mu}{\mu (J_1)}.$$
Define $\psi_1$ on the remainder of the tower consistently 
% $$\psi_1:S_1^{h_1-1}(J_1\setminus [J_1^{\prime}\cup J_1^*])\cup 
% R_1^{h_1-1}I^{\prime}_1\to R^{h_1-1}(J_1)$$
such that 
\begin{eqnarray*} 
\psi_1(x)= 
\left\{\begin{array}{ll}
S_1^{i}\circ \psi_1 \circ S_1^{-i}(x) & \mbox{if $x\in S_1^i(J_1\setminus J_1^{\prime})$ for $0\leq i<h_1$} \\ 
S_1^{i}\circ \psi_1 \circ R_1^{-i}(x) & \mbox{if $x\in R_1^{i}(I^{\prime}_1)$ for $0\leq i<h_1$} \\ 
\beta_1^i\circ \psi_1 \circ \beta_1^{-i}(x) & \mbox{if $x\in J_1^*(i)$ for $0\leq i<h_1$}
\end{array}
\right.
\end{eqnarray*}
In this case, define $S_2:Y_2\to Y_2$ such that 
$S_2=\psi_1^{-1}\circ S_1\circ \psi_1$. Note 
\begin{eqnarray*} 
S_2(x)= 
\left\{\begin{array}{ll}
R_1(x) & \mbox{if $x\in R_1^{i}I^{\prime}_1$ for $0\leq i<h_1-1$} \\ 
S_1(x) & \mbox{if $x\in S_1^{i}(J_1\setminus J^{\prime}_1)$ for $0\leq i<h_1-1$} \\
\beta_1(x) & \mbox{if $x\in J_1^*(i)$ for $0\leq i<h_1-1$} \\ 
\psi_1^{-1}\circ S_1\circ \psi_1(x) & \mbox{if $x\in Y_1^{\prime}\cup S_1^{h_1-1}(J_1\setminus J_1^{\prime})\cup R_1^{h_1-1}I^{\prime}_1\cup \beta_1^{h_1-1}J_1^*$} 
\end{array}
\right.
\end{eqnarray*}
and $S_2$ is isomorphic to $S_1$ and $S$. 
\end{case}

\begin{case}[$d<0$]  
Define $\tau_1:X_1^{\prime}\to X_1^*$ as a measure preserving map between 
normalized spaces 
$(X_1^{\prime},\mathbb{B} \cap X_1^{\prime},\frac{\mu}{\mu (X_1^{\prime})})$ 
and $(X_1^*,\mathbb{B} \cap X_1^*,\frac{\mu}{\mu (X_1^*)})$. 
Extend $\tau_1$ to the new tower base, 
$$\tau_1:[I_1\setminus I_1^{\prime}]\cup I_1^*\cup J^{\prime}_1\to I_1$$
such that $\tau_1$ preserves normalized measure between 
$$\frac{\mu}{\mu ([I_1\setminus I_1^{\prime}]\cup I_1^*\cup J^{\prime}_1)}\mbox{ and }\frac{\mu}{\mu (I_1)}.$$
Define $\tau_1$ on the remainder of the tower consistently 
% $$\tau_1:R_1^{h_1-1}(I_1\setminus [I_1^{\prime}\cup I_1^*])\cup 
% S_1^{h_1-1}J^{\prime}_1\to R^{h_1-1}(I_1)$$
such that 
\begin{eqnarray*} 
\tau_1(x)= 
\left\{\begin{array}{ll}
R_1^{i}\circ \tau_1 \circ R_1^{-i}(x) & \mbox{if $x\in R_1^i(I_1\setminus I_1^{\prime})$ for $0\leq i<h_1$} \\ 
R_1^{i}\circ \tau_1 \circ S_1^{-i}(x) & \mbox{if $x\in S_1^{i}(J^{\prime}_1)$ for $0\leq i<h_1$} \\ 
\alpha_1^i\circ \tau_1 \circ \alpha_1^{-i}(x) & \mbox{if $x\in I_1^*(i)$ for $0\leq i<h_1$}
\end{array}
\right.
\end{eqnarray*}
In this case, define $R_2:X_2\to X_2$ such that 
\begin{eqnarray*} 
R_2(x)= 
\left\{\begin{array}{ll}
S_1(x) & \mbox{if $x\in S_1^{i}J^{\prime}_1$ for $0\leq i<h_1-1$} \\ 
R_1(x) & \mbox{if $x\in R_1^{i}(I_1\setminus I^{\prime}_1)$ for $0\leq i<h_1-1$} \\
\alpha_1(x) & \mbox{if $x\in I_1^*(i)$ for $0\leq i<h_1-1$} \\ 
\tau_1^{-1}\circ R_1\circ \tau_1(x) & \mbox{if $x\in X_1^{\prime}\cup R_1^{h_1-1}(I_1\setminus I_1^{\prime})\cup S_1^{h_1-1}J^{\prime}_1\cup \alpha_1^{h_1-1}I_1^*$} 
\end{array}
\right.
\end{eqnarray*}
Clearly, $R_2$ is isomorphic to $R_1$ and $R$. 

Define $\psi_1:Y_1^{\prime}\to Y_1^*$ as a measure preserving map between 
normalized spaces 
$(Y_1^{\prime},\mathbb{B} \cap Y_1^{\prime},\frac{\mu}{\mu (Y_1^{\prime})})$ 
and $(Y_1^*,\mathbb{B} \cap Y_1^*,\frac{\mu}{\mu (Y_1^*)})$. 
Extend $\psi_1$ to the new tower base, 
$$\psi_1:[J_1\setminus (J_1^{\prime}\cup J_1^*)]\cup I^{\prime}_1\to J_1$$
such that $\psi_1$ preserves normalized measure between 
$$\frac{\mu}{\mu ([J_1\setminus (J_1^{\prime}\cup J_1^*)]\cup I^{\prime}_1)}\mbox{ and }\frac{\mu}{\mu (J_1)}.$$
Define $\psi_1$ on the remainder of the tower consistently 
% $$\psi_1:S_1^{h_1-1}(J_1\setminus [J_1^{\prime}\cup J_1^*])\cup 
% R_1^{h_1-1}I^{\prime}_1\to R^{h_1-1}(J_1)$$
such that 
\begin{eqnarray*} 
\psi_1(x)= 
\left\{\begin{array}{ll}
S_1^{i}\circ \psi_1 \circ S_1^{-i}(x) & \mbox{if $x\in S_1^i(J_1\setminus [J_1^{\prime}\cup J_1^*])$ for $0\leq i<h_1$} \\ 
S_1^{i}\circ \psi_1 \circ R_1^{-i}(x) & \mbox{if $x\in R_1^{i}(I^{\prime}_1)$ for $0\leq i<h_1$}
\end{array}
\right.
\end{eqnarray*}
Define $S_2:Y_2\to Y_2$ such that $S_2=\psi_1^{-1}\circ S_1\circ \psi_1$. Note 
\begin{eqnarray*} 
S_2(x)= 
\left\{\begin{array}{ll}
R_1(x) & \mbox{if $x\in R_1^{i}(I^{\prime}_1)$ for $0\leq i<h_1-1$} \\ 
S_1(x) & \mbox{if $x\in S_1^{i}(J_1\setminus [J^{\prime}_1\cup J_1^*])$ for $0\leq i<h_1-1$}
\end{array}
\right.
\end{eqnarray*}
Transformation $S_2$ is isomorphic to $S_1$ and $S$. 
\end{case}
% End definition of $S_2$
Define $T_2$ as 
\begin{eqnarray*} 
T_2(x)= 
\left\{\begin{array}{ll}
R_2(x) & \mbox{if $x\in X_2$} \\ 
S_2(x) & \mbox{if $x\in Y_2$}
\end{array}
\right.
\end{eqnarray*}
Clearly, neither $T_1$ nor $T_2$ are ergodic. For $T_1$, $X_1$ and $Y_1$ 
are ergodic components, and $X_2$, $Y_2$ are ergodic components for 
$T_2$. See the appendix for a pictorial of the multiplexing operation 
used to produce $R_2$ and $S_2$ from $R_1$, $S_1$ and the intermediary 
maps defined in this section. 
\subsection{General Multiplexing Operation}
For $n\geq 1$, suppose that $R_n$ and $S_n$ have been defined on $X_n$ and $Y_n$ respectively. Construct Rohklin towers of height $h_n$ for each $R_n$ and $S_n$, and such that $I_n$ is the base of the $R_n$ tower, $J_n$ is the base of the $S_n$ tower, and $\mu (\bigcup_{i=0}^{h_n-1}R_n^iI_n) + \mu (\bigcup_{i=0}^{h_n-1}S_n^iJ_n)>1-\epsilon_n$. Let $I^{\prime}_n \subset I_n$ be such that $\mu (I^{\prime}_n) = r_n \mu (I_n)$. Similarly, suppose $J^{\prime}_n \subset J_n$ such that $\mu (J^{\prime}_n) = s_n \mu (J_n)$. 

We define $R_{n+1}$ and $S_{n+1}$ by switching the subcolumns 
$$\{I^{\prime}_n, R_n(I^{\prime}_n),R_n^2(I^{\prime}_n),\ldots ,R_n^{h_n-1}(I^{\prime}_n)\}$$ and 
$$\{J^{\prime}_n, S_n(J^{\prime}_n),S_n^2(J^{\prime}_n),\ldots ,S_n^{h_n-1}(J^{\prime}_n)\}.$$ Let 
\begin{eqnarray*}
X_{n+1}&=&[\bigcup_{i=0}^{h_n-1}R_n^i(I_n\setminus I_n^{\prime})]
\cup [\bigcup_{i=0}^{h_n-1}S_n^iJ_n^{\prime}] 
\cup [X_n\setminus \bigcup_{i=0}^{h_n-1}R_n^iI_n] \\ 
Y_{n+1}&=&[\bigcup_{i=0}^{h_n-1}S_n^i(J_n\setminus J_n^{\prime})]
\cup [\bigcup_{i=0}^{h_n-1}R_n^iI_n^{\prime}] 
\cup [Y_n\setminus \bigcup_{i=0}^{h_n-1}S_n^iJ_n].
\end{eqnarray*}
As in the initial case, it may be necessary to transfer measure 
between each column and its respective residual. 
We can follow the same algorithm as above, and define 
maps $\tau_n, \alpha_n, \psi_n$ and $\beta_n$. 
Thus, we get the following definitions: 
\begin{case}[$d\geq 0$]
\begin{eqnarray*}
\tau_n(x)=& 
\left
\{\begin{array}{ll}
R_n^{i}\circ \tau_n \circ R_n^{-i}(x) & \mbox{if $x\in R_n^i(I_n\setminus I_n^{\prime})$ for $0\leq i<h_n$} \\ 
R_n^{i}\circ \tau_n \circ S_n^{-i}(x) & \mbox{if $x\in S_n^{i}(J^{\prime}_n\setminus I_1^*)$ for $0\leq i<h_n$}
\end{array}
\right.
\end{eqnarray*}
\begin{eqnarray*} 
R_{n+1}(x)=& 
\left\{\begin{array}{ll}
S_n(x) & \mbox{if $x\in S_n^{i}(J^{\prime}_n\setminus I_n^*)$ for $0\leq i<h_n-1$} \\ 
R_n(x) & \mbox{if $x\in R_n^{i}(I_n\setminus I^{\prime}_n)$ for $0\leq i<h_n-1$} \\
\tau_n^{-1}\circ R_n\circ \tau_n(x) & \mbox{if $x\in X_n^{\prime}\cup R_n^{h_n-1}(I_n\setminus I_n^{\prime})\cup S_n^{h_n-1}(J^{\prime}_n\setminus I_n^*)$} 
\end{array}
\right.
\end{eqnarray*}
and $R_{n+1}=\tau_n^{-1}\circ R_n\circ \tau_n$.
\begin{eqnarray*} 
\psi_n(x)=& 
\left\{\begin{array}{ll}
S_n^{i}\circ \psi_n \circ S_n^{-i}(x) & \mbox{if $x\in S_n^i(J_n\setminus J_n^{\prime})$ for $0\leq i<h_n$} \\ 
S_n^{i}\circ \psi_n \circ R_n^{-i}(x) & \mbox{if $x\in R_n^{i}(I^{\prime}_n)$ for $0\leq i<h_n$} \\ 
\beta_n^i\circ \psi_n \circ \beta_n^{-i}(x) & \mbox{if $x\in J_n^*(i)$ for $0\leq i<h_n$}
\end{array}
\right.
\end{eqnarray*}
\begin{eqnarray*} 
S_{n+1}(x)=& 
\left\{\begin{array}{ll}
R_n(x) & \mbox{if $x\in R_n^{i}I^{\prime}_n$ for $0\leq i<h_n-1$} \\ 
S_n(x) & \mbox{if $x\in S_n^{i}(J_n\setminus J^{\prime}_n)$ for $0\leq i<h_n-1$} \\
\beta_n(x) & \mbox{if $x\in J_n^*(i)$ for $0\leq i<h_n-1$} \\ 
\psi_n^{-1}\circ S_n\circ \psi_n(x) & \mbox{if $x\in Y_n^{\prime}\cup S_n^{h_n-1}(J_n\setminus J_n^{\prime})\cup R_n^{h_n-1}I^{\prime}_n\cup \beta_n^{h_n-1}J_n^*$} 
\end{array}
\right.
\end{eqnarray*}
and $S_{n+1}=\psi_n^{-1}\circ S_n\circ \psi_n$.
\end{case}

\begin{case}[$d<0$]
\begin{eqnarray*} 
\tau_n(x)=& 
\left\{\begin{array}{ll}
R_n^{i}\circ \tau_n \circ R_n^{-i}(x) & \mbox{if $x\in R_n^i(I_n\setminus I_n^{\prime})$ for $0\leq i<h_n$} \\ 
R_n^{i}\circ \tau_n \circ S_n^{-i}(x) & \mbox{if $x\in S_n^{i}(J^{\prime}_n)$ for $0\leq i<h_n$} \\ 
\alpha_n^i\circ \tau_n \circ \alpha_n^{-i}(x) & \mbox{if $x\in I_n^*(i)$ for $0\leq i<h_n$}
\end{array}
\right.
\end{eqnarray*}
\begin{eqnarray*} 
R_{n+1}(x)=& 
\left\{\begin{array}{ll}
S_n(x) & \mbox{if $x\in S_n^{i}J^{\prime}_n$ for $0\leq i<h_n-1$} \\ 
R_n(x) & \mbox{if $x\in R_n^{i}(I_n\setminus I^{\prime}_n)$ for $0\leq i<h_n-1$} \\
\alpha_n(x) & \mbox{if $x\in I_n^*(i)$ for $0\leq i<h_n-1$} \\ 
\tau_n^{-1}\circ R_n\circ \tau_n(x) & \mbox{if $x\in X_n^{\prime}\cup R_n^{h_n-1}(I_n\setminus I_n^{\prime})\cup S_n^{h_n-1}J^{\prime}_n\cup \alpha_n^{h_n-1}I_n^*$} 
\end{array}
\right.
\end{eqnarray*}
and $R_{n+1}=\tau_n^{-1}\circ R_n\circ \tau_n$.
\begin{eqnarray*} 
\psi_n(x)=& 
\left\{\begin{array}{ll}
S_n^{i}\circ \psi_n \circ S_n^{-i}(x) & \mbox{if $x\in S_n^i(J_n\setminus [J_n^{\prime}\cup J_n^*])$ for $0\leq i<h_n$} \\ 
S_n^{i}\circ \psi_n \circ R_n^{-i}(x) & \mbox{if $x\in R_n^{i}(I^{\prime}_n)$ for $0\leq i<h_n$}
\end{array}
\right.
\end{eqnarray*}
\begin{eqnarray*} 
S_{n+1}(x)=& 
\left\{\begin{array}{ll}
R_n(x) & \mbox{if $x\in R_n^{i}(I^{\prime}_n)$ for $0\leq i<h_n-1$} \\ 
S_n(x) & \mbox{if $x\in S_n^{i}(J_n\setminus [J^{\prime}_n\cup J_n^*])$ for $0\leq i<h_n-1$} \\
\psi_n^{-1}\circ S_n\circ \psi_n(x) & \mbox{if $x\in Y_n^{\prime}\cup S_n^{h_n-1}(J_n\setminus [J_n^{\prime}\cup J_n^*])\cup R_n^{h_n-1}(I^{\prime}_n)$} 
\end{array}
\right.
\end{eqnarray*}
and $S_{n+1}=\psi_n^{-1}\circ S_n\circ \psi_n$.
\end{case}
\subsection{The Limiting Transformation}
Define the transformation 
$T_{n+1}:X_{n+1}\cup Y_{n+1}\to X_{n+1}\cup Y_{n+1}$ such that 
\begin{eqnarray*} 
T_{n+1}(x)= 
\left\{\begin{array}{ll}
R_{n+1}(x) & \mbox{if $x\in X_{n+1}$} \\ 
S_{n+1}(x) & \mbox{if $x\in Y_{n+1}$} 
\end{array}
\right.
\end{eqnarray*}
The set where $T_{n+1}\neq T_n$ is determined by the top levels 
of the Rokhlin towers, the residual and the transfer set. 
Note the transfer set has measure $d$. Since this set is used to adjust 
the size of the residuals between stages, it can be bounded below 
a constant multiple of $\epsilon_n$. Thus, there is a fixed constant $\kappa$, 
independent of $n$, such that 
$T_{n+1}(x)=T_n(x)$ except for $x$ in a set of measure less than 
$\kappa (\epsilon_n + {1}/{h_n})$. Since $\sum_{n=1}^{\infty} (\epsilon_n + {1}/{h_n}) < \infty$, 
$T(x) = \lim_{n\to \infty}T_n(x)$ exists almost everywhere, 
and preserves normalized Lebesgue measure. 
Without loss of generality, we may assume $\kappa$ and $h_n$ 
are chosen such that if 
$$E_n=\{x\in X| T_{n+1}(x)\neq T_n(x)\}$$
then $\mu (E_n) < \kappa \epsilon_n$ for $n\in \natural$. 
In the following section, additional structure and conditions 
are implemented to ensure that 
$T$ inherits properties from $R$ and $S$, and is also ergodic. 

For the remainder of this paper, 
assume the parameters are chosen such that 
\begin{enumerate}
\item $\lim_{n\to \infty}r_n= 0$;
\item $\sum_{n=1}^{\infty}r_n=\sum_{n=1}^{\infty}s_n=\infty$; 
\item $\lim_{n\to \infty}\mu (Y_n)=0$;
\item $\sum_{n=1}^{\infty}\epsilon_n<\infty$.
\end{enumerate}

\subsection{Isomorphism Chain Consistency}
In the following sections, rigidity and ergodicity will be established 
on sets from a refining sequence of partitions. 
For $n\in \natural$, let $P_n$ be a refining sequence of finite partitions 
which generates the sigma algebra. By refining $P_n$ further if necessary, 
assume $X_n,Y_n,X_n^*,Y_n^*\in P_n$. 
Also, assume 
$R_n^i(I^{\prime}_n), R_n^i(I_n\setminus I^{\prime}_n), S_n^i(J^{\prime}_n), 
S_n^i(J_n\setminus J^{\prime}_n)$ are elements of $P_n$ 
for $0\leq i<h_n$. 
Finally, assume for $0\leq i<h_n-1$, if $p\in P_n$ and 
$p\subset R_n^i(I_n)$ then $R_n(p)\in P_n$. 
Likewise, assume for $0\leq i<h_n-1$, if $p\in P_n$ and 
$p\subset S^i(J_n)$ then $S_n(p)\in P_n$. 
Previously, we required that $\tau_n$ map certain finite orbits from 
the $R_n$ and $S_n$ towers to a corresponding orbit in the $R_{n+1}$ tower. 
In this section, further regularity is imposed on $\tau_n$ relative to $P_n$ 
to ensure dynamical properties of $R_n$ are inherited by $R_{n+1}$. 

Let $P_n^{\prime}=
\{p\in P_n| p\subset \bigcup_{i=0}^{h_n-1}R_n^i(I_n\setminus I_n^{\prime})\}$. 
For each of the following three cases, impose the corresponding 
restriction on $\tau_n$:
\begin{enumerate}
\item for $d_R=0$ and $p\in P_n^{\prime}$, $\tau_n$ is the identity map (i.e. $\tau_n(p)=p$);
\item for $d_R>0$ and $p\in P_n^{\prime}$, $\tau_n(p)\subset p$;
\item for $d_R<0$ and $p\in P_n^{\prime}$, $p\subset \tau_n(p)$.
\end{enumerate}
This can be accomplished by uniformly distributing the appropriate mass 
from the sets $R_n^i(I_n^*)$ using $\tau_n$. 
Note that $\tau_n$ either preserves Lebesgue measure in the case $d_R=0$, 
or $\tau_n$ contracts sets relative to Lebesgue measure in the case $d_R>0$, 
or it inflates measure in the case $d_R<0$. In all three cases, 
for $p\in P_n^{\prime}$,
$$
\frac{\mu (p)}{\mu (\tau_n (p))}=\frac{\mu (X_{n+1})}{\mu (X_n)}.
$$
It is straightforward to verify for any set $A$ measurable relative to $P_n^{\prime}$, 
$$
\mu (A\triangle \tau_nA)<|\frac{\mu (X_{n+1})}{\mu (X_n)} - 1|.
$$
The properties of $\tau_n$ allow approximation of $R_{n+1}$ by $R_n$ indefinitely over time. 
This is needed to establish our rigidity sequence for the limiting transformation $T$. 
This lemma is not required for establishing ergodicity, but for convenience we will reuse it 
to prove our limiting $T$ is ergodic. 
\begin{lemma}
\label{rescalinglem}
Suppose $\delta >0$ and $n\in \natural$ is chosen such that 
\begin{eqnarray*}
|\frac{\mu (X_{n+1})}{\mu (X_n)} - 1| < \frac{\delta}{7}, \\ 
r_n + \epsilon_n + \mu (Y_n) < \frac{\delta}{7}.
\end{eqnarray*}
Then for $A,B\in P_n$ and $i\in \natural$, the following holds:
\begin{enumerate}
\item $|\mu (R_{n+1}^iA\cap B)-\mu (A)\mu (B)| < |\mu (R_{n}^iA\cap B)-\mu (A)\mu (B)| + \delta$;
\item $\mu (R_{n+1}^iA\triangle A) < \mu (R_{n}^iA\triangle A) + \delta$.
\end{enumerate}
\end{lemma}
\begin{proof} 
For $A,B\in P_n$, let 
$A^{\prime}=\bigcup_{p\in P_n^{\prime}}p\cap A$ and 
$B^{\prime}=\bigcup_{p\in P_n^{\prime}}p\cap B$. 
Since $\mu (\bigcup_{j=0}^{h_n-1}R_n^j(I_n^{\prime}))=h_n\mu (I_n^{\prime}) < r_n$ and 
$\mu (X_n^*)<\epsilon_n$, then 
$\mu (A\triangle A^{\prime})<r_n+\epsilon_n < \frac{\delta}{7}$. 
Likewise, $\mu (B\triangle B^{\prime})<\frac{\delta}{7}$. 
Since $|\frac{\mu (X_{n+1})}{\mu (X_n)} - 1| < \frac{\delta}{7}$, then 
$\mu (A\triangle \tau_nA)<\frac{\delta}{7}$. By applying the triangle inequality 
several times, we can get our approximations. Below is a sequence of 
quantities to chain through such that consecutive values in the chain are less than 
$\delta /7$ apart. 
\begin{eqnarray*}
\mu (R_{n+1}^iA\cap B)\rightarrow \mu (R_{n+1}^iA\cap B^{\prime})\rightarrow \mu (R_{n+1}^iA^{\prime}\cap B^{\prime})=\mu (\tau_n^{-1}R_{n}^i\tau_nA^{\prime}\cap B^{\prime}) \\
\mu (\tau_n^{-1}R_{n}^i\tau_nA^{\prime}\cap B^{\prime})\rightarrow \mu (R_{n}^i\tau_nA^{\prime}\cap \tau_n B^{\prime})\rightarrow \mu (R_{n}^i\tau_nA^{\prime}\cap B^{\prime}) \\ 
\rightarrow \mu (R_{n}^iA^{\prime}\cap B^{\prime}) 
\rightarrow \mu (R_{n}^iA^{\prime}\cap B) \rightarrow \mu (R_{n}^iA\cap B) 
\end{eqnarray*}
Each arrow in the chain signifies less than $\frac{\delta}{7}$ difference. Hence, 
$$|\mu (R_{n+1}^iA\cap B)-\mu (R_{n}^iA\cap B)|<\delta$$ which implies 
$$|\mu (R_{n+1}^iA\cap B)-\mu (A)\mu (B)| < |\mu (R_{n}^iA\cap B)-\mu (A)\mu (B)| + \delta.$$

The second part of the lemma can be proven in a similar fashion using the triangle inequality, or chaining through the following six approximations. 
\begin{eqnarray*}
\mu(R_{n+1}^iA\triangle A)\rightarrow \mu (R_{n+1}^iA\triangle A^{\prime})\rightarrow 
\mu (R_{n+1}^iA^{\prime}\triangle A^{\prime})=
\mu (\tau_n^{-1}R_n^i \tau_nA^{\prime}\triangle A^{\prime}) \\ 
\mu (\tau_n^{-1}R_n^i \tau_nA^{\prime}\triangle A^{\prime})\rightarrow \mu (R_n^i \tau_nA^{\prime}\triangle \tau_nA^{\prime}) \rightarrow \mu (R_n^i \tau_nA^{\prime}\triangle A^{\prime}) \\ 
\rightarrow \mu (R_n^i A^{\prime}\triangle A^{\prime}) 
\rightarrow \mu (R_n^i A^{\prime}\triangle A) \rightarrow \mu (R_n^i A\triangle A)
\end{eqnarray*}
Since each arrow indicates a difference less than $\frac{\delta}{7}$, then 
$$|\mu(R_{n+1}^iA\triangle A) - \mu(R_{n}^iA\triangle A)| < \delta.$$ 
This completes the proof of the lemma. 
\end{proof}

\section{Establishing Rigidity}
Suppose that $\rho_n$ is a rigidity sequence for $R$. 
In this section, we define parameters such that $T$ is rigid on $\rho_n$. 
\subsection{Waiting for Rigidity}
Let $\delta_n$ be a sequence 
of positive real numbers such that $\lim_{n\to \infty}\delta_n=0$. 
Since $T_n|_{X_n}=R_n$ is rigid, choose 
natural number $M_n^1 > \max{\{h_{n-1},M_{n-1}^1\}}$ such that for $N\geq M_n^1$, 
and $A\in P_{n-1}\cap X_n$, 
$$\mu (R_n^{\rho_N}A\triangle A)<\delta_n.$$
Choose $\epsilon_{n}$ such that 
\begin{eqnarray}
\frac{\epsilon_{n-1}}{M_n^1}<\epsilon_{n}.
\end{eqnarray} 
Also, without loss of generality, assume $h_n > M_n^1$. 
Below we show this choice of $\epsilon_{n}$ is sufficient to produce 
$T(x)=\lim_{n\to \infty}T_n(x)$ rigid on $\rho_n$. 
First, we provide a diagram and heuristic description of our method 
for establishing rigidity on $\rho_n$. 

\subsection{The Key Idea}
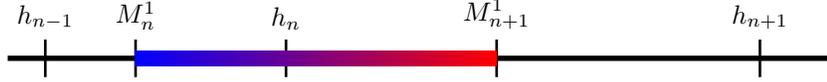
\begin{tikzpicture}
{[line width=2pt] 
{[black] \draw (1+.5,2) -- (1+2.2,2);}}
{[line width=4pt] 
{[blue]\draw (1+2.2,2) -- (1+4.2,2);}}
{[line width=4pt] 
{[red]\draw (1+4.2,2) -- (1+7,2);}}
{[line width=2pt] 
{[black]\draw (1+7,2) -- (1+11.5,2);}}
%\draw (1+.5,.75) -- (1+2.25,.75);
{[line width=1pt]
\draw[-] (1+1,1.75) -- (1+1,2.25) node[, above] {$h_{n-1}$};
\draw[-] (2.2+1,1.75) -- (2.2+1,2.25) node[, above] {$M_n^1$};
\draw[-] (4.2+1,1.75) -- (4.2+1,2.25) node[, above] {$h_{n}$};
\draw[-] (7+1,1.75) -- (7+1,2.25) node[, above] {$M_{n+1}^1$};
\draw[-] (10.5+1,1.75) -- (10.5+1,2.25) node[, above] {$h_{n+1}$};
\shade[left color=blue, right color=red] (3.2,1.9) rectangle (8,2.1);
}
\end{tikzpicture}

\noindent 
To establish rigidity of $T$, we can focus on the 
asymptotic rigidity of $T$ on the intervals $(M_n^1,M_{n+1}^1]$. 
We have chosen $M_n^1$ sufficiently large such that rigidity 
\mbox{"kicks in"} for $R_n$ and $\rho_i > M_n^1$. 
Lemma \ref{rescalinglem} allows us to approximate $R_n$ by $R_{n+1}$ 
as $\rho_i$ becomes closer to $M_{n+1}^1$. The fact that we can 
choose $\epsilon_{n+1}$ arbitrarily small compared to ${1}/{M_{n+1}^1}$ 
allows us to carry over the approximation to $T$. 
A precise proof is given below. 

\subsection{Rigidity Proof}
If $E_{n+1}=\{x\in X: T_{n+2}(x)\neq T_{n+1}(x)\}$ and 
$$E_{n+1}^{1}=\bigcup_{i=0}^{M_{n+1}^1-1}[T_{n+2}^{-i}E_{n+1}\cup T_{n+1}^{-i}E_{n+1}]$$
then $\mu (E_{n+1}^{1}) < 2M_{n+1}^1 \kappa \epsilon_{n+1}$. 
For $x\notin E_{n+1}^{1}$, $T_{n+2}^i(x)=T_{n+1}^i(x)$ for $0\leq i\leq M_{n+1}^1$. 
Let $\hat{E}_{n+1}^{1} = \bigcup_{k=n+1}^{\infty}E_k^{1}$. 
For $x\notin \hat{E}_{n+1}^{1}$ and $0\leq i\leq M_{n+1}^1$, 
$T^i(x) = T_{n+1}^i(x)$. 
Also,
$$\mu (\hat{E}_{n+1}^{1}) < \sum_{k=n+1}^{\infty} 2M_{k}^1\kappa \epsilon_k 
< \sum_{k=n+1}^{\infty}2\kappa \epsilon_{k-1} \to 0$$
as $n\to \infty$. 

\begin{proof}[Proof of rigidity]
Let $A$ be a set in $P_{n_1}$ for some $n_1$, and let $\delta >0$. 
Choose $n_2\geq n_1$ such that for $n\geq n_2$, 
\begin{enumerate}
\item $|\frac{\mu (X_{n+1})}{\mu (X_n)} - 1|< \delta / 28$;
\item $r_n + \epsilon_n + \mu (Y_n) < \delta / 28$;
\item $\delta_n < \delta / 6$;
\item $\sum_{i=n_2}^{\infty}2\kappa \epsilon_{i} < \delta / 12$.
\end{enumerate}
For $n>n_2$, let $M_n^1<N\leq M_{n+1}^1$, $A_1 = A\setminus \hat{E}_{n+1}^{1}$ 
and $A_2=A\cap X_n$. Thus, \begin{eqnarray*}
\mu (T^{\rho_N}A\triangle A) &\leq& \mu (T^{\rho_N}A\triangle T^{\rho_N}A_1) 
+ \mu (T^{\rho_N}A_1\triangle A) \\ 
&=& \mu (A\triangle A_1)  + \mu (R_{n+1}^{\rho_N}A_1\triangle A) \\ 
&<& \frac{\delta}{4}  + \mu (R_{n+1}^{\rho_N}A_1\triangle R_{n+1}^{\rho_N}A) + \mu (R_{n+1}^{\rho_N}A\triangle A) \\ 
&<& \frac{\delta}{2} + \mu (R_{n+1}^{\rho_N}A\triangle A).
\end{eqnarray*}
By Lemma \ref{rescalinglem}, 
\begin{eqnarray*}
\mu (T^{\rho_N}A\triangle A) 
&<& \frac{\delta}{2} + \mu (R_{n+1}^{\rho_N}A\triangle A) < \frac{3\delta}{4} + \mu (R_{n}^{\rho_N}A\triangle A) \\
&\leq& \frac{3\delta}{4} + \mu (R_{n}^{\rho_N}A\triangle R_{n}^{\rho_N}A_2) + 
\mu (R_{n}^{\rho_N}A_2\triangle A_2) + \mu (A_2\triangle A) \\ 
&<& \frac{3\delta}{4} + 2\mu (Y_n) + \delta_n < \delta.
\end{eqnarray*}
Therefore, $\rho_n$ is a rigidity sequence for $T$. 
\end{proof}

\section{Ergodicity}
A measure preserving transformation $T$ on a Lebesgue space is 
ergodic if any invariant set has measure zero or one. 
It is well known this is equivalent to the mean and pointwise 
ergodic theorem. 
For our purposes, we use the following equivalent condition 
of ergodicity: for all measurable sets $A$ and $B$, 
$$\lim_{N\to \infty}\frac1N\sum_{i=0}^{N-1}\mu (T^iA\cap B)=\mu(A)\mu(B).$$
Let $P_n, n\in \natural$ be a sequence of finite refining partitions 
as defined in the previous section. 
Using approximation, $T$ is ergodic if the previous condition holds 
for all natural numbers $n$ and sets $A$ and $B$ from $P_n$.

\subsection{Ergodic Parameter Choice}
Let $\delta_n$ be a sequence 
of positive real numbers such that $\lim_{n\to \infty}\delta_n=0$. 
Since $T_n|_{X_n}=R_n$ is ergodic, choose 
natural number $M_n=M_n^{2}$ such that for $N\geq M_n$, and 
sets $A,B\in P_n\cap X_n$,
$$|\frac1N\sum_{i=0}^{N-1}\frac{\mu (T_n^iA\cap B)}{\mu (X_n)}-\frac{\mu(A)\mu(B)}{\mu (X_n)^2}|<\delta_n.$$
Note that 
\begin{eqnarray*}
|\frac1N\sum_{i=0}^{N-1}\mu (T_n^iA\cap B)&-&\mu(A)\mu(B)| \\
&=& \mu (X_n)|\frac1N\sum_{i=0}^{N-1}\frac{\mu (T_n^iA\cap B)}{\mu (X_n)}-\frac{\mu(A)\mu(B)}{\mu (X_n)}| \\ 
&\leq& \mu (X_n)|\frac1N\sum_{i=0}^{N-1}\frac{\mu (T_n^iA\cap B)}{\mu (X_n)}-\frac{\mu(A)\mu(B)}{\mu (X_n)^2}| \\ 
&+& |\frac{\mu(A)\mu(B)}{\mu (X_n)} - \mu(A)\mu(B)| \\ 
&<& \delta_n + \frac{\mu (Y_n)}{\mu (X_n)}
\end{eqnarray*}
Choose $\epsilon_{n+1}$ such that 
\begin{eqnarray}
\frac{\epsilon_n}{M_n}<\epsilon_{n+1}.
\end{eqnarray} 
\subsection{Approximation}

As previously, set $E_{n+1}=\{x\in X: T_{n+2}(x)\neq T_{n+1}(x)\}$. Let 
$$E_{n+1}^{2}=\bigcup_{i=0}^{M_{n+1}-1}[T_{n+2}^{-i}E_{n+1}\cup T_{n+1}^{-i}E_{n+1}]$$
Thus, $\mu (E_{n+1}^{2}) < 2M_{n+1} \kappa \epsilon_{n+1}$. 
For $x\notin E_{n+1}^{2}$, $T_{n+2}^i(x)=T_{n+1}^i(x)$ for $0\leq i\leq M_{n+1}$. 
Let $\hat{E}_{n+1}^{2} = \bigcup_{k=n+1}^{\infty}E_k^{2}$. 
For $x\notin \hat{E}_{n+1}^{2}$ and $0\leq i\leq M_{n+1}$, 
$T^i(x) = T_{n+1}^i(x)$. 
Also,
$$\mu (\hat{E}_{n+1}^{2}) < \sum_{k=n+1}^{\infty} 2M_{k}\kappa \epsilon_k 
< \sum_{k=n+1}^{\infty}2\kappa \epsilon_{k-1} \to 0$$
as $n\to \infty$. 

\begin{proof}[Proof of ergodicity]
Let $A$ and $B$ be sets in $P_{n_1}$ for some $n_1$, and let $\delta >0$. 
Choose $n_2\geq n_1$ such that for $n\geq n_2$, 
\begin{enumerate}
\item $|\frac{\mu (X_{n+1})}{\mu (X_n)} - 1|< \delta / 28$;
\item $r_n + \epsilon_n + \mu (Y_n) < \delta / 28$;
\item $\delta_n + \frac{\mu (Y_n)}{\mu (X_n)}< \delta / 4$;
\item $\sum_{i=n_2}^{\infty}2\kappa \epsilon_{i} < \delta / 12$.
\end{enumerate}
For $n>n_2$, let $M_n<N\leq M_{n+1}$, $A_1 = A\setminus \hat{E}_{n+1}^2$ and 
$B_1 = B\setminus \hat{E}_{n+1}^2$. 
\begin{eqnarray*}
|\frac{1}{N}\sum_{i=0}^{N-1}\mu (T^iA\cap B) &-& \mu (A)\mu (B)| \\ 
&\leq& |\frac{1}{N}\sum_{i=0}^{N-1}\mu (T^iA\cap B) - 
\frac{1}{N}\sum_{i=0}^{N-1}\mu (T^iA_1\cap B)| \\ 
&+& |\frac{1}{N}\sum_{i=0}^{N-1}\mu (T^iA_1\cap B)-\mu (A)\mu (B)| \\ 
&\leq&\frac{1}{N}\sum_{i=0}^{N-1}|\mu (T^iA)-\mu (T^iA_1)| \\ 
&+& |\frac{1}{N}\sum_{i=0}^{N-1}\mu (R_{n+1}^iA_1\cap B)-\mu (A)\mu (B)| \\ 
&<& \frac{\delta}{4} + |\frac{1}{N}\sum_{i=0}^{N-1}\mu (R_{n+1}^iA_1\cap B)-\mu (R_{n+1}^iA\cap B)| \\ 
&+& |\frac{1}{N}\sum_{i=0}^{N-1}\mu (R_{n+1}^iA\cap B)-\mu (A)\mu (B)| \\ 
&<& \frac{\delta}{2} + |\frac{1}{N}\sum_{i=0}^{N-1}\mu (R_{n+1}^iA\cap B)-\mu (A)\mu (B)|.
\end{eqnarray*}
Since $A,B\in P_n$, then by Lemma \ref{rescalinglem}, 
\begin{eqnarray*}
|\frac{1}{N}\sum_{i=0}^{N-1}\mu (T^iA\cap B) &-& \mu (A)\mu (B)| \\ 
&<& \frac{\delta}{2} + \frac{1}{N}\sum_{i=0}^{N-1}|\mu (R_{n+1}^iA\cap B)-\mu (A\cap B)| \\
&<& 
\frac{3\delta}{4} + \frac{1}{N}\sum_{i=0}^{N-1}|\mu (R_{n}^iA\cap B)-\mu (A\cap B)| < \delta.
\end{eqnarray*}
Since $\delta$ is chosen arbitrarily, and the above holds for any 
$n>n_2$ and $M_n<N\leq M_{n+1}$, then $T$ is ergodic.
\end{proof}

\section{Weak Mixing}
Since the weak mixing component is dissipative, and the resulting 
transformation inherits its rigidity properties from $R$, we do not 
focus on multiplexing with general weak mixing transformations. 
Instead, we set $S$ equal to the famous Chacon transformation. 
It is defined via cutting and stacking, and considered the 
earliest construction demonstrated to be weak mixing and not mixing. 
See \cite{Fri70} for a precise definition. 
For the remainder of this paper, assume both $R$ and $S$ are defined 
on $([0,1),\mu ,\mathbb{B})$ where $\mu$ is Lebesgue measure. In this section, 
we further specify $h_n$ and switching sets 
$C_{n}=\bigcup_{i=0}^{h_n-1}R_n^i(I_n^{\prime})$ for $n\in \natural$. 
As in previous sections, all conditions imposed are easily satisfied 
by choosing a faster growing sequence of tower heights $h_n$. 
No upper bounds are imposed on the growth rate of $h_n$. 
\subsection{Switching Set Definition}
For each $k\in \natural$ and $n>k$, denote $U_k^n=\bigcup_{j=k}^{n-1}C_j$, 
$V_k^n=(U_k^n)^c$ and $\dot{V}_k^n=V_k^n \cap X_n$. 
Since $R_n$ is ergodic on $X_n$, $r_n$ is fixed, 
and $C_n$ predominantly represents long orbits of $R_n$, 
then $h_n$ may be chosen sufficiently large 
such that $C_n$ is near conditionally independent 
of $\dot{V}_k^n$ for each $k<n$. 

\noindent
Precisely, define $h_n$ and $C_n$ such that 
\begin{eqnarray}
\label{SwitchingSetInd}
|\frac{\mu (C_n\cap \dot{V}_k^n)}{\mu (X_n)} - 
\frac{\mu (C_n)\mu (\dot{V}_k^n)}{\mu (X_n)^2}| \leq 
\frac{1}{2} \mu (C_n)\mu (\dot{V}_k^n).
\end{eqnarray}
\begin{lemma}
For each $k\in \natural$, $\lim_{n\to \infty}\mu (V_{k}^n)=0$. 
\end{lemma}
\begin{proof} Suppose the claim is not true, and 
there exists $k_0\in \natural$ such that 
$$\lim_{n\to \infty}\mu (V_k^n)>0.$$ 
Since $\lim_{n\to \infty}\mu (Y_n)=0$, we can choose $k_1>k_0$ such that 
$\mu (Y_j) < \frac12 \mu(V_{k_1}^{k_1+n})$ for $j\geq k_1$ and 
$n\in \natural$. 
Thus, 
\begin{eqnarray*}
\frac{\mu (C_{k_1+1}\cap \dot{V}_{k_1}^{k_1+1})}{\mu (X_{k_1+1})} 
&\geq& \frac{\mu (C_{k_1+1}) \mu ( \dot{V}_{k_1}^{k_1+1})}{\mu (X_{k_1+1})^2}
- \frac12 \mu (C_{k_1+1})\mu (\dot{V}_{k_1}^{k_1+1}) \\ 
\mu (C_{k_1+1}\cap \dot{V}_{k_1}^{k_1+1}) 
&\geq& \mu (V_{k_1}^{k_1+1}) \frac{\mu (V_{k_1}^{k_1+1}\cap X_{k_1+1})}{\mu (V_{k_1}^{k_1+1})} \mu (C_{k_1+1})[\frac{1}{\mu (X_{k_1+1})} 
- \frac{\mu (X_{k_1+1})}{2}] \\
&>& \frac14 \mu (C_{k_1+1})\mu (V_{k_1}^{k_1+1}).
\end{eqnarray*}
Hence, 
\begin{eqnarray*}
\mu (V_{k_1}^{k_1+2}) &=& 
\mu (V_{k_1}^{k_1+1}) - \mu (C_{k_1+1}\cap V_{k_1}^{k_1+1}) \\ 
&<& \mu (V_{k_1}^{k_1+1})[1 - \frac14 \mu (C_{k_1+1})] \\ 
&<& (1 - \frac14 \mu (C_{k_1}))(1 - \frac14 \mu (C_{k_1+1})).
\end{eqnarray*}
Extending this inductively produces 
$$
\mu (V_{k_1}^{k_1+n}) < \prod_{i=0}^{n-1}(1-\frac14 \mu (C_{k_1+i})).
$$
Note that 
$$
\mu (C_n)=\mu (I_n^{\prime})h_n = 
\frac{\mu (I_n^{\prime})}{\mu (I_n)}\mu (I_n)h_n = r_n\mu (X_n).
$$
Since $\sum_{n=1}^{\infty}r_n=\infty$ and 
$\lim_{n\to \infty} \mu (X_n)=1$, then both 
$\sum_{n=1}^{\infty}\mu (C_n)=\infty$ and 
$\sum_{n=1}^{\infty}\frac14 \mu (C_n)=\infty$. 
This is sufficient to force 
$\lim_{n\to \infty}\prod_{i=0}^{n-1}(1-\frac14 \mu (C_{k_1+i}))=0$ 
which proves our claim by contradiction. 
\end{proof}
The previous claim establishes the following property 
that almost every point falls in infinitely many sets $C_n$. 
\begin{property}
\label{limsup}
$\mu (\bigcap_{n=1}^{\infty}\bigcup_{i=n}^{\infty}C_i)=1$. 
\end{property}
\subsection{Multiplexing Chacon's Transformation}
Chacon's transformation $S$ is typically defined using 
cutting and stacking \cite{Fri70}. Initialize 
$I_1^0=[0,2/3)$ and $\mathcal{C}_1=I_1^0$. 
Cut $I_1$ into 3 pieces of equal width, $I_2^0=[0,2/9), I_2^1=[2/9,4/9), 
I_2^3=[4/9,2/3)$, and add a single spacer $I_2^2=[2/3, 8/9)$ above 
interval $I_2^1$. Stack into a single column 
$\mathcal{C}_2=<I_2^0,I_2^1,I_2^2,I_2^3>$. 
Define $S$ as the linear map from $I_2^i$ to $I_2^{i+1}$ for $i=0,1,2$. 
Let $H_n={(3^n-1)}/{2}$ be the height 
of column $\mathcal{C}_n$. Obtain $\mathcal{C}_{n+1}$ by cutting $\mathcal{C}_n$ into 3 subcolumns 
of equal width, $\mathcal{C}_n^0,\mathcal{C}_n^1,\mathcal{C}_n^2$, adding one spacer above the 
second subcolumn and stacking left to right. Again, $S$ maps each 
level linearly to the level directly above it. Also, 
notice the height of $\mathcal{C}_{n+1}$ equals $H_{n+1}=3H_n+1=\frac{3^{n+1}-1}{2}$. 
The main property we utilize in this work is related to one of its 
limit joinings. 

\begin{lemma}
\label{ChaLem}
Let $S$ be Chacon's transformation. Given any two measurable sets, 
$A$ and $B$, 
$$
\lim_{n\to \infty} \mu (S^{H_n}A\cap B)=(\mu (A\cap B)+\mu (S^{-1}A\cap B))/2.
$$
\end{lemma}
\begin{proof} Each column $\mathcal{C}_n$, $n\in \natural$, has a single level of spacer 
above precisely half the mass of the top level of $\mathcal{C}_n$. This includes 
the spacers added when $\mathcal{C}_n$ is cut into 3 subcolumns, as well as the 
infinitely many spacers added when $\mathcal{C}_{n+1}, \mathcal{C}_{n+2},\ldots$ are cut 
into 3 subcolumns and stacked. Thus, $S^{H_n}$ maps half of each level 
to the same level, and maps the other half to the level directly below itself. 
This establishes the lemma for sets consisting of a finite union of levels. 
Since the levels of the columns form a refining sequence of partitions 
which generate the sigma algebra, the lemma follows by approximation. 
\end{proof}
\subsection{Weak Mixing Stage}
Now we define $S_n$ inductively to ensure 
the final transformation $T$ is weak mixing. Let $S_1$ be the Chacon 
transformation defined on $Y_1$. Suppose $S_n\simeq S$ has been defined on 
$Y_n$. Now we specify the manner in which $S_{n+1}$ should be defined. 
\subsubsection{Local Approximation of Switching Sets}
Choose natural number $k_n>n$ such that for each $i=0,1,\ldots ,h_n-1$, 
there exists a finite collection of indices $\hat{K}_n^i$ and dyadic intervals 
$K_n^i(j)$, $j\in \hat{K}_n^i$, such that $\mu (K_n^i(j))=\frac{1}{2^{k_n}}$ 
and $K_n^i=\bigcup_{j\in \hat{K}_n^i}K_n^i(j)$ satisfies 
$\mu (R_n^iI_n^{\prime}\triangle K_n^i)<
(\frac{\epsilon_n}{h_n})^2\mu (I_n^{\prime})$. 
Let $\hat{G}_n^i=\{j\in \hat{K}_n^i:\mu (R_n^iI_n^{\prime}\cap K_n^i(j))>
(1-\frac{\epsilon_n}{h_n})\mu (K_n^i(j))\}$. 
It is not difficult to show 
$\mu (\bigcup_{j\in \hat{G}_n^i}K_n^i(j))>(1-\frac{\epsilon_n}{h_n})\mu (I_n^{\prime})$.  Set $G_n^i=\bigcup_{j\in \hat{G}_n^i}K_n^i(j)$. 
For each $n\in \natural$, define 
$$
D_n=\bigcup_{\ell=0}^{h_n-1}G_n^{\ell}.
$$
Note that 
$$
\mu (C_n\setminus D_n)<
\sum_{\ell=0}^{h_n-1}\frac{\epsilon_n}{h_n} =\epsilon_n.
$$
Next, we show almost every point falls in infinitely many $D_n$. 
\begin{property}
\label{limsup2}
$\mu (\overline{D})=1$ where $\overline{D}=\bigcap_{n=1}^{\infty}\bigcup_{i=n}^{\infty}D_i$. 
\end{property}
\begin{proof} Given $\epsilon >0$, choose $N=N(\epsilon)\in \natural$ 
such that $\sum_{n=N}^{\infty}\epsilon_n<\epsilon$. 
Thus, 
\begin{eqnarray*}
\mu (\bigcup_{n=N}^{\infty}D_n) &\geq& \mu (\bigcup_{n=N}^{\infty}C_n) 
- \sum_{n=N}^{\infty}\mu (C_n\setminus D_n) \\
&>& 1 - \sum_{n=N}^{\infty}\epsilon_n > 1 - \epsilon.
\end{eqnarray*}
Since $\epsilon$ is arbitrarily small, then 
$\mu (\bigcup_{n=N}^{\infty}D_n)=1$, and Property \ref{limsup2} is establshed.
\end{proof}
\subsubsection{Weak Mixing Component}
The main goal in this work is to demonstrate how properties of a given 
ergodic transformation can be transferred to produce a tailored weak mixing 
transformation. Since the weak mixing component will dissipate over time, 
we do not focus on introducing general properties using $S$. 
Instead, we set $S$ to the Chacon transformation inside our towerplex 
construction. Thus, $S_n$ will be isomorphic to Chacon's transformation. 
By Lemma \ref{ChaLem}, for each $n\in \natural$, there exists 
$m_n\in \natural$ 
such that for each $i=0,1,\ldots ,h_n-1$, $j\in \hat{K}_n^i$ and 
$A=K_n^i(j)$, 
$$
|\mu (S_{n+1}^{H_{m_n}}A\cap A)-\frac12 \mu (A)|<\epsilon_n \mu(A)
$$
and 
$$
|\mu (S_{n+1}^{H_{m_n}}A\cap S^{-1}(A)-\frac12 \mu (A)|<\epsilon_n \mu(A).
$$

\noindent 
Let $w_n=\min{\{\mu (K_{\ell}^i(j))>0:1\leq \ell \leq n, 0\leq i\leq h_n-1, 
j\in \hat{K}_{\ell}^i\}}$. 
Choose $h_{n+1}$ such that 
\begin{eqnarray}
h_{n+1} > \frac{H_{m_n}}{\epsilon_n w_n}.
\end{eqnarray}
\section{Slow Weak Mixing Theorem}
In this final section, we prove our main result using the towerplex constructions. 
First, we give explicit parameters $r_n$ and $s_n$ that 
can be used to generate our rigid weak mixing examples. 
Let $r_n=\frac{\mu (I_n^{\prime})}{\mu (I_n)}=
\frac{1}{2(n+2)}$ and 
$s_n=\frac{\mu (J_n^{\prime})}{\mu (J_n)}=
\frac{1}{2}$.
Thus, each of the switching sets have measure 
$\mu (\bigcup_{i=0}^{h_n-1}R_n^i(I_n^{\prime}))=
{(\mu (X_n)-\mu (X_n^*))}/{2(n+2)}$ 
and 
$\mu (\bigcup_{i=0}^{h_n-1}S_n^i(J_n^{\prime}))=
{(\mu (Y_n)-\mu (Y_n^*))}/2$ 
for $n\in \natural$.  
This implies 
$$\mu (Y_{n+1})=\frac{1}{2(n+2)} [(n+1)\mu (Y_n) + 1] + \kappa_n \epsilon_n$$
where $|\kappa_n|$ is bounded for all $n\in \natural$. 
If all residuals had zero mass, then 
$\kappa_n \epsilon_n = 0$ and by induction: 
$$
\mu (X_n)=\frac{n}{n+1}\mbox{  and  }\mu (Y_n)=\frac{1}{n+1}. 
$$
In the case the residuals are not null, the next lemma obtains
$$
\lim_{n\to \infty}\mu (X_n)=1\mbox{,  }\lim_{n\to \infty}\mu (Y_n)=0.
$$
Parameters given here are called the canonical towerplex settings. 
\begin{lemma}
\label{CanParLem}
If real numbers $\epsilon_n>0$ are chosen sufficiently small for $n\in \natural$,
then a canonical towerplex construction, given by $r_n=\frac{1}{2(n+2)}$ 
and $s_n=\frac{1}{2}$ has the property, for $n\in \natural$, 
\begin{eqnarray}
\label{Ydecay}
\frac{1}{n+2} < \mu (Y_n) < \frac{1}{n}.
\end{eqnarray}
\end{lemma}
\begin{proof} 
The function $f(y)=({1}/{2(n+2)})[(n+1)y + 1]$ has a fixed point at $y={1}/{(n+3)}$. 
If $y > {1}/{(n+3)}$, then $f(y) > {1}/{(n+3)}$. 
Thus, if $\epsilon_n$ is sufficiently small, and $\mu (Y_n) > {1}/{(n+2)}$, 
then $\mu (Y_{n+1}) > {1}/{(n+3)}$. This establishes the first inequality 
from (\ref{Ydecay}). 

To prove the second inequality, assume $y=\mu (Y_n) < {1}/{n}$ for fixed $n\in \natural$. 
Thus, 
\begin{eqnarray*}
f(y) &<& \frac{1}{2(n+2)} [(n+1)(\frac{1}{n}) + 1]  = \frac{1}{2(n+2)} [2 + \frac{1}{n}]  \\ 
&=& \frac{1}{n+2} + \frac{1}{2n(n+2)}  = \frac{1}{n+1} + \frac{1-n}{2n(n+1)(n+2)} 
\leq \frac{1}{n+1}. 
\end{eqnarray*}
Therefore, if $\epsilon_n$ is sufficiently small, $\mu (Y_{n+1}) <  \frac{1}{n+1}$.
\end{proof}
Now, we are ready to prove our main theorem. 
\begin{theorem}
\label{towerplex}
Given an ergodic measure preserving transformation $R$ on a Lebesgue 
probability space, and a rigid sequence $\rho_n$ for $R$, there exists 
a weak mixing transformation $T$ on a Lebesgue probability space such that 
$T$ is rigid on $\rho_n$. 
\end{theorem}
\begin{proof} Much of the details have been established in the previous 
sections. In particular, the conditions imposed 
in each of the sections on ergodicity, rigidity and weak mixing, are 
consistent.  Essentially, $\epsilon_n\to 0$ arbitrarily fast which 
is possible since only the extra mass from successive Rokhlin towers 
is bounded by $\epsilon_n$. Also, each section imposes a lower bound 
on the growth rate of the tower heights $h_n$, but no upper bound. 
Appendix \ref{conditions} lists conditions that can be 
used to support the explicit proofs. 
Below, we need to complete the argument that $T$ is weak mixing. 

Suppose $f\neq 0$ is an eigenfunction for $T$ with eigenvalue $\lambda$. 
Since we established that $T$ is ergodic, we may assume 
$|f|$ is a constant. Without loss of generality, assume $|f|=|\lambda|=1$. 
Given $\delta >0$, there exists a set 
$\Lambda_{\delta}$ of positive measure such that for $x,y\in \Lambda_{\delta}$, 
$|f(x)-f(y)|<\delta$. Let $\Lambda_{\delta}^{\prime}$ be the 
set of Lebesgue density points of $\Lambda_{\delta}$. 
In particular, if 
$\Lambda_{\delta}^{\prime}=\{x\in \Lambda_{\delta}: 
\lim_{\eta \to 0}\frac{\mu (\Lambda_{\delta}\cap (x-\eta,x+\eta))}{2\eta}=1\}$,
then 
$\mu (\Lambda_{\delta}^{\prime})=\mu (\Lambda_{\delta})>0$. 
Choose $x\in \Lambda_{\delta}^{\prime}\cap \overline{D}$. 
Choose $\eta^{\prime}>0$ such that for $\eta <\eta^{\prime}$, 
$\frac{\mu (\Lambda_{\delta}\cap (x-\eta,x+\eta))}{2\eta}>1-\delta$. 
Choose $n\in \natural$ such that $\frac{1}{2^{k_n}}<\eta^{\prime}$, 
$\sum_{i=n}^{\infty}\epsilon_i <\delta$ and $x\in D_n$. 
There exists $i=i(x)$ such that $x\in G_n^i$, and subsequently 
$j=j(x)$ such that $x\in K_n^i(j)$. 
Let $\eta_x = \max{\{|y-x|:y\in K_n^i(j)\}}$. Note $\eta_x <\eta^{\prime}$, 
and $\frac{\mu (\Lambda_{\delta}\cap (x-\eta_x,x+\eta_x))}{2\eta_x}>1-\delta$.
Thus, 
\begin{eqnarray}
\mu (\Lambda_{\delta}\cap K_n^i(j))&>&\mu (K_n^i(j))-2\eta_x \delta 
\mu (K_n^i(J)) - 2\delta \mu(K_n^i(j)) \\ 
&\geq& (1-2\delta)\mu (K_n^i(j)). 
\end{eqnarray}
Hence, 
\begin{eqnarray*}
&&|\mu (S_{n+1}^{H_{m_n}}(\Lambda_{\delta}\cap K_n^i(j))\cap (\Lambda_{\delta}\cap K_n^i(j))) - \frac12 \mu (\Lambda_{\delta}\cap K_n^i(j))| \\
&\leq& |\mu (S_{n+1}^{H_{m_n}}(\Lambda_{\delta}\cap K_n^i(j))\cap (\Lambda_{\delta}\cap K_n^i(j))) - \mu (S_{n+1}^{H_{m_n}}(K_n^i(j))\cap K_n^i(j))| \\ 
&+& |\mu (S_{n+1}^{H_{m_n}}(K_n^i(j))\cap (K_n^i(j))) - \frac12 \mu (K_n^i(j))| \\ 
&+& |\frac12 \mu (K_n^i(j)) - 
\frac12 \mu (\Lambda_{\delta}\cap K_n^i(j))| \\ 
&<& 4\delta \mu (K_n^i(j)) + \epsilon_n\mu (K_n^i(j)) + \delta \mu (K_n^i(j)) = (5\delta +\epsilon_n)\mu (K_n^i(j)).
\end{eqnarray*}
We wish to establish that $T$ is weak mixing, and $T$ does not equal $S_{n+1}$ 
everywhere.  In particular, $T$ may differ from $S_{n+1}$ on the top levels 
of the towers of height $h_{n+1}, h_{n+2}, \ldots$, on 
the accompanying residuals, and on the transfer sets. 
However, we have chosen the growth 
of the tower heights sufficient to ensure the set where $T$ and $S_{n+1}$ 
may differ will be small relative to interval, $K_n^i(j)$. 
Thus, 
\begin{eqnarray}
\mu (\{x\in Y_{n+1}:Tx\neq S_{n+1}x\})&<&
\sum_{i=n+1}^{\infty}[\frac{1}{h_{i}} + 4\epsilon_{i}] \\ 
&<& \sum_{i=n}^{\infty}[\frac{5\epsilon_iw_n}{H_{m_i}+1}].
\end{eqnarray}
This implies 
\begin{eqnarray*}
\mu (\{x\in Y_{n+1}&:&T^ix\neq S_{n+1}^ix, i=1,2,\ldots ,H_{m_n}+1\}) \\ 
&<& w_n(H_{m_n}+1)\sum_{i=n}^{\infty}\frac{5\epsilon_i}{H_{m_i}+1} 
<5w_n\sum_{i=n}^{\infty}\epsilon_i < 5\delta w_n.
\end{eqnarray*}
Hence,
\begin{eqnarray*}
&&|\mu (T^{H_{m_n}}(\Lambda_{\delta}\cap K_n^i(j))\cap (\Lambda_{\delta}\cap K_n^i(j))) - \frac12 \mu (\Lambda_{\delta}\cap K_n^i(j))| \\
&\leq& 
|\mu (T^{H_{m_n}}(\Lambda_{\delta}\cap K_n^i(j))\cap (\Lambda_{\delta}\cap K_n^i(j))) \\ 
&-& \mu (S_{n+1}^{H_{m_n}}(\Lambda_{\delta}\cap K_n^i(j))\cap (\Lambda_{\delta}\cap K_n^i(j)))| \\ 
&+&
|\mu (S_{n+1}^{H_{m_n}}(\Lambda_{\delta}\cap K_n^i(j))\cap (\Lambda_{\delta}\cap K_n^i(j))) - \frac12 \mu (\Lambda_{\delta}\cap K_n^i(j))| \\ 
&<& 5\delta w_n + (5\delta +\epsilon_n)\mu (K_n^i(j)) \leq 
(10\delta +\epsilon_n)\mu (K_n^i(j)).
\end{eqnarray*}

For $\delta$ and $\epsilon$ sufficiently small, 
there exists $x_1\in \Lambda_{\delta}\cap K_n^i(j)$ such that 
$T^{H_{m_n}}x_1\in \Lambda_{\delta}\cap K_n^i(j)$, and 
$x_2\in \Lambda_{\delta}\cap K_n^i(j)$ such that 
$T^{H_{m_n}+1}x_2\in \Lambda_{\delta}\cap K_n^i(j)$. 
Thus, 
\begin{eqnarray}
|\lambda^{H_{m_n}}f(x_1)-f(x_1)|=|f(T^{H_{m_n}}x_1)-f(x_1)| < \delta,
\end{eqnarray}
and
\begin{eqnarray}
|\lambda^{H_{m_n}+1}f(x_2)-f(x_2)|=|f(T^{H_{m_n}+1}x_2)-f(x_2)| < \delta.
\end{eqnarray}
Hence, 
$$
|\lambda^{H_{m_n}}-1| < \frac{\delta}{|f(x_1)|} = \delta\mbox{  and  }
|\lambda^{H_{m_n}+1}-1| < \frac{\delta}{|f(x_2)|} = \delta.
$$
Therefore, 
$$
|\lambda -1|=
|\lambda^{H_{m_n}+1}-\lambda^{H_{m_n}}| \leq 
|\lambda^{H_{m_n}+1}-1| + 
|\lambda^{H_{m_n}}-1| < 2\delta.
$$
Since $\delta >0$ may be chosen arbitrarily small, then $\lambda =1$. 
Since it was established that $T$ is ergodic in an earlier section, 
then $f$ must be a constant. Therefore, $T$ is weak mixing. 
\end{proof}
Our theorem establishes the following corollaries 
which answer questions raised in the ground-breaking works 
\cite{BdJLR} and \cite{EisGri}.
\begin{corollary}
\label{discrete}
Given any ergodic measure preserving transformation $R$ on a Lebesgue 
probability space with discrete spectrum, 
and a rigidity sequence $\rho_n$ for $R$, 
there exists a weak mixing transformation $T$ with 
rigidity sequence $\rho_n$. In particular, for any 
$k\in \natural, k\geq 2$, there exists a weak mixing transformation 
with $k^n$, $n\in \natural$, as a rigidity sequence. 
\end{corollary}
The next corollary gives an explicit characterization of "large" rigid sequences 
for weak mixing transformations. 
While this corollary appears known in \cite{AarHosLem12}, 
our characterization gives a general concrete method for establishing 
"large" rigidity sequences of weak mixing transformations. 
Given a sequence $\mathcal{A}$, 
define the density function $g_{\mathcal{A}}:\natural \to [0,1]$ 
such that $g_{\mathcal{A}}(k)={\#(\mathcal{A} \cap \{1,2,\ldots k\})} / {k}$. 
\begin{corollary}
\label{slowrigid}
Given any real-valued function $f:\natural \to (0,\infty)$ such that 
$$\lim_{n\to \infty}f(n)=0,$$ there exists a weak mixing transformation 
with rigidity sequence $\mathcal{A}$ such that 
$$
\lim_{n\to \infty}\frac{f(n)}{g_{\mathcal{A}}(n)}=0.
$$
Also, there exist weak mixing transformations with rigidity sequences 
$\rho_n$ satisfying 
$$
\lim_{n\to \infty}\frac{\rho_{n+1}}{\rho_n}=1.
$$
\end{corollary}
\begin{proof} Let $\alpha$ be an irrational number and 
$R_{\alpha}$ the rotation by $2\pi \alpha$ on the unit circle. 
Given $\epsilon >0$, define $\mathcal{A}(\epsilon )=
\{ j\in \natural:| \exp{(2\pi \alpha j)} -1|<\epsilon \}$, and 
for $n\in \natural$, define  $\mathcal{A}(\epsilon, n )=
\mathcal{A}(\epsilon )\cap \{1,2,\ldots n\}$. 
For $\bar{\epsilon} = \{\epsilon_1<\epsilon_2<\ldots \}$, let 
$\mathcal{A}(\bar{\epsilon}) =\bigcup_{n=1}^{\infty}\mathcal{A}(\epsilon_n, n)$. 
If $\lim_{n\to \infty}\epsilon_n=0$ and $\mathcal{A}(\bar{\epsilon})$ is infinite, 
then $\mathcal{A}(\bar{\epsilon})$ forms a rigidity sequence for $R_{\alpha}$. 
Let $f:\natural \to (0,\infty)$ be such that $\lim_{n\to \infty}f(n)=0$. 
Since $\mathcal{A} (1/2^i)$ has positive density for $i\in \natural$, 
there exists $j_i\in \natural$ 
such that 
$$
\frac{|\mathcal{A}( {1}/{2^i}, j )|}{j} > f(j). 
$$
for all $j\geq j_i$. 
For $k\in \natural$, choose $i=i_k\in \natural$ such that 
$j_i +1\leq k\leq j_{i+1}$. Set $\epsilon_k = \frac{1}{2^i}$ 
and let $\mathcal{A}=\mathcal{A}(\bar{\epsilon})$. Thus, 
\begin{eqnarray}
g_{\mathcal{A}}(k)&=& \frac{|\mathcal{A} \cap \{1,2,\ldots ,k\}|}{k} \\ 
&\geq& \frac{|\mathcal{A}(\epsilon_{k},k)|}{k} > 2^i f(k). 
\end{eqnarray}
Hence, 
\begin{eqnarray}
\frac{f(k)}{g_{\mathcal{A}}(k)} &>& \frac{1}{2^i}. 
\end{eqnarray}
This confirms that $\lim_{k\to \infty} {f(k)}/{g_{\mathcal{A}}(k)}=0$ 
for the rigidity sequence $\mathcal{A}$. 
Therefore, by Theorem \ref{towerplex}, $\mathcal{A}$ is a rigidity sequence 
for a weak mixing transformation.  
The second assertion of Corollary \ref{slowrigid} can be established 
in a similar manner. Since ergodic rotations on the unit circle have 
rigid sequences $\rho_{n}$ such that 
$\lim_{n\to \infty} {\rho_{n+1}}/{\rho_n} = 1$, then weak mixing transformations 
admit these rigid sequences as well. 
\end{proof}
Previously, it was established that denominators from convergents 
of continued fractions serve as rigidity sequences for 
weak mixing transformations. A partial result was provided 
in \cite{EisGri} for restricted convergents, and then a general 
result was established in \cite{BdJLR}. 
In this paper, we extend these results to show that 
any rigidity sequence for an ergodic rotation on the unit circle is also 
a rigidity sequence for a weak mixing transformation. 
This includes sequences $q_n$ formed from the denominators 
of convergents ${p_n}/{q_n}$ of an irrational $\alpha$. 
\begin{corollary}
\label{rotations}
Let $\alpha\in (0,1)$ be any irrational number, and let 
$\rho_n$ be a sequence of natural numbers satisfying 
$$
\lim_{n\to \infty} |\exp{(2\pi i\alpha \rho_n)}-1|=0.
$$
Then there exists a weak mixing transformation $T$ such that 
$\rho_n$ is a rigidity sequence for $T$. 
\end{corollary}

\noindent
{\bf\Large  Acknowledgements} \\
%The author acknowledges the D.C. Beltway for providing countless hours of driving (slowly) in a big loop with plenty of time for %ergodic things $\ddot\smile$  
The author wishes to thank Joseph Rosenblatt, Andrew Parrish, Ayse Sahin, Karl Petersen, Nathaniel Friedman, Cesar Silva, Keri Kornelson and Tatjana Eisner for feedback on a previous version of this article. 

\appendix 
\section{Towerplex Pictorial}
This appendix provides an illustration of towers 
for $R_1$, $S_1$, and the multiplexing operation applied to obtain
towers for $R_2$ and $S_2$.  The picture below represents only the 
case where $d_R > 0$ and $d_S < 0$. The other cases are 
handled as described in the section on towerplex constructions. 
Also, the general case of deriving $R_{n+1}$ and $S_{n+1}$ 
from $R_n$ and $S_n$ is analogous to the initial multiplexing operation 
for deriving $R_2$ and $S_2$. 

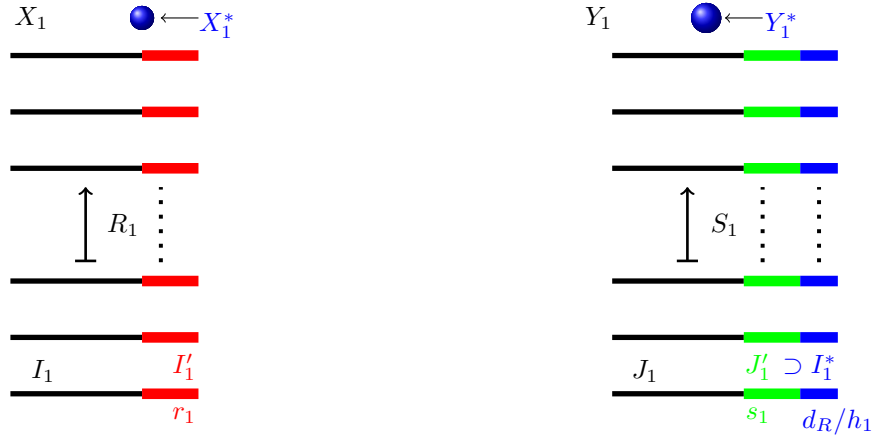
\begin{tikzpicture}
{[line width=2pt]
{[black]
\draw (1+.5,0) -- (1+2.25,0) node[near start,above] {$I_1$};
\draw (1+.5,.75) -- (1+2.25,.75);
\draw (1+.5,1.5) -- (1+2.25,1.5);
{[line width=1pt]
\draw[|->] (1+1.50,1.75) -- (1+1.50,2.75);
}
\node at (1+2,2.25) {$R_1$};
\draw (1+.5,3) -- (1+2.25,3);
\draw (1+.5,3.75) -- (1+2.25,3.75);
\draw (1+.5,4.5) -- (1+2.25,4.5);
}
}
{[line width=4pt]
{[red]
\draw (1+2.25,0) -- (1+3,0) node[near end,below]{$r_1$} node[near end,above]{$I_1^{\prime}$};
\draw (1+2.25,.75) -- (1+3,.75);
\draw (1+2.25,1.5) -- (1+3,1.5);
{[line width=1.5pt]
{[black]
\draw[loosely dotted] (1+2.5,1.75) -- (1+2.5,2.75);
}
}
\draw (1+2.25,3) -- (1+3,3);
\draw (1+2.25,3.75) -- (1+3,3.75);
\draw (1+2.25,4.5) -- (1+3,4.5);
}
}
\shade [ball color=blue] (1+2.25,5) circle (.16) node[](residual){};
\draw[->] (1+3.0,5)--(1+2.5,5) node[blue](myarrow1){};
\node[blue] at (1+3.25,4.9) {$X_1^*$};
\node[black,left=of residual]{$X_1$};
% end R tower
% start S tower
{[line width=2pt]
{[black]
\draw (1+8.5,0) -- (1+10.25,0) node[near start,above] {$J_1$};
\draw (1+8.5,.75) -- (1+10.25,.75);
\draw (1+8.5,1.5) -- (1+10.25,1.5);
{[line width=1pt]
\draw[|->] (1+9.50,1.75) -- (1+9.50,2.75);
}
\node at (1+10,2.25) {$S_1$};
\draw (1+8.5,3) -- (1+10.25,3);
\draw (1+8.5,3.75) -- (1+10.25,3.75);
\draw (1+8.5,4.5) -- (1+10.25,4.5);
}
}
{[line width=4pt]
{[green]
\draw (1+10.25,0) -- (1+11,0) node[near start,below] {$s_1$} node[near start,above]{$J_1^{\prime}$};
\draw (1+10.25,.75) -- (1+11,.75);
\draw (1+10.25,1.5) -- (1+11,1.5);
{[line width=1.5pt]
{[black]
\draw[loosely dotted] (1+10.5,1.75) -- (1+10.5,2.75);
}
}
\draw (1+10.25,3) -- (1+11,3);
\draw (1+10.25,3.75) -- (1+11,3.75);
\draw (1+10.25,4.5) -- (1+11,4.5);
}
}
{[line width=4pt]
{[blue]
\draw (1+11,0) -- (1+11.5,0) node[below]{$d_R/h_1$} node[near start,above] {$\supset I_1^*$};
\draw (1+11,.75) -- (1+11.5,.75);
\draw (1+11,1.5) -- (1+11.5,1.5);
{[line width=1.5pt]
{[black]
\draw[loosely dotted] (1+11.25,1.75) -- (1+11.25,2.75);
}
}
\draw (1+11,3) -- (1+11.5,3);
\draw (1+11,3.75) -- (1+11.5,3.75);
\draw (1+11,4.5) -- (1+11.5,4.5);
}
}
\shade [ball color=blue] (1+9.75,5) circle (.20) node[](residual){};
\draw[->] (1+10.5,5)--(1+10.0,5) node[blue](myarrow1){};
\node[blue] at (1+10.75,4.9) {$Y_1^*$};
\node[black,left=of residual]{$Y_1$};
\end{tikzpicture}

Transformations $R_2$ and $S_2$ are derived from $R_1$ and $S_1$ by 
switching the red subcolumn with the green subcolumn. We refer to 
these sets as the switching sets, and are the main multiplexing operation. 
In order to preserve maps isomorphic to $R$ and $S$, and avoid 
redefining $R_1$ or $S_1$ on most of the probability space, it may 
be necessary to transfer measure between the towers and residuals. 
This is a rescaling operation, and these sets are referred to as 
transfer sets. In the case where $d_R>0$, the blue colored subcolumn $I_1^*$ 
from $J_1^{\prime}\subset Y_1$ is absorbed into $X_1^{\prime}$. For $d_S < 0$, 
mass is removed from $Y_1^*$ and added as a blue subcolumn to define 
$S_2$.

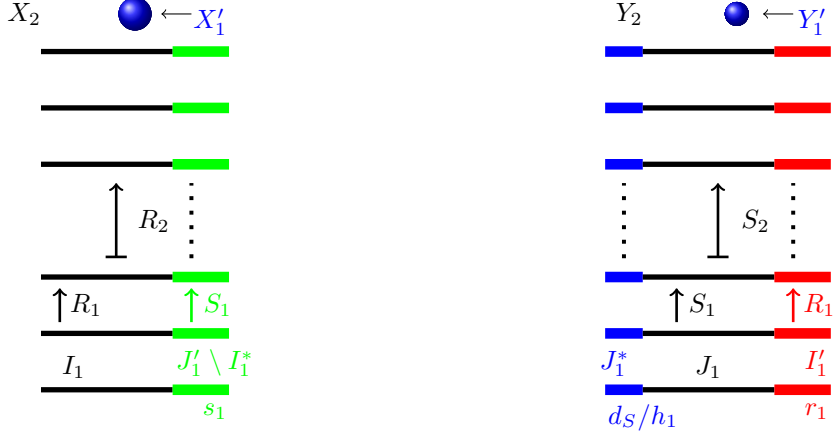
\begin{tikzpicture}
{[line width=2pt]
{[black]
\draw (1+.5,0) -- (1+2.25,0) node[near start,above] {$I_1$};
\draw (1+.5,.75) -- (1+2.25,.75);
{[line width=1pt]
\draw[->] (1+.75,.9) -- (1+.75,1.35);
}
\node at (1+1.1,1.125) {$R_1$};
\draw (1+.5,1.5) -- (1+2.25,1.5);
{[line width=1pt]
\draw[|->] (1+1.50,1.75) -- (1+1.50,2.75);
}
\node at (1+2,2.25) {$R_2$};
\draw (1+.5,3) -- (1+2.25,3);
\draw (1+.5,3.75) -- (1+2.25,3.75);
\draw (1+.5,4.5) -- (1+2.25,4.5);
}
}
{[line width=4pt]
{[green]
\draw (1+2.25,0) -- (1+3,0) node[near end,below] {$s_1$} node[near end,above]{$J_1^{\prime}\setminus{I_1^*}$};
\draw (1+2.25,.75) -- (1+3,.75);
{[line width=1pt]
\draw[->] (1+2.5,.9) -- (1+2.5,1.35);
}
\node at (1+2.85,1.125) {$S_1$};
\draw (1+2.25,1.5) -- (1+3,1.5);
{[line width=1.5pt]
{[black]
\draw[loosely dotted] (1+2.5,1.75) -- (1+2.5,2.75);
}
}
\draw (1+2.25,3) -- (1+3,3);
\draw (1+2.25,3.75) -- (1+3,3.75);
\draw (1+2.25,4.5) -- (1+3,4.5);
}
}
\shade [ball color=blue] (1+1.75,5) circle (.22) node[](residual){};
\draw[->] (1+2.5,5)--(1+2.1,5) node[blue](myarrow1){};
\node[blue] at (1+2.75,4.9) {$X_1^{\prime}$};
\node[black,left=of residual]{$X_2$};
% end R tower
% start S tower
{[line width=4pt]
{[blue]
\draw (1+8,0) -- (1+8.5,0) node[below]{$d_S/h_1$} node[near start,above] {$J_1^*$};
\draw (1+8,.75) -- (1+8.5,.75);
\draw (1+8,1.5) -- (1+8.5,1.5);
{[line width=1.5pt]
{[black]
\draw[loosely dotted] (1+8.25,1.75) -- (1+8.25,2.75);
}
}
\draw (1+8,3) -- (1+8.5,3);
\draw (1+8,3.75) -- (1+8.5,3.75);
\draw (1+8,4.5) -- (1+8.5,4.5);
}
}
{[line width=2pt]
{[black]
\draw (1+8.5,0) -- (1+10.25,0) node[midway,above] {$J_1$};
\draw (1+8.5,.75) -- (1+10.25,.75);
{[line width=1pt]
\draw[->] (1+8.95,.9) -- (1+8.95,1.35);
}
\node at (1+9.3,1.125) {$S_1$};
\draw (1+8.5,1.5) -- (1+10.25,1.5);
{[line width=1pt]
\draw[|->] (1+9.50,1.75) -- (1+9.50,2.75);
}
\node at (1+10,2.25) {$S_2$};
\draw (1+8.5,3) -- (1+10.25,3);
\draw (1+8.5,3.75) -- (1+10.25,3.75);
\draw (1+8.5,4.5) -- (1+10.25,4.5);
}
}
{[line width=4pt]
{[red]
\draw (1+10.25,0) -- (1+11,0) node[near end,below] {$r_1$} node[near end,above]{$I_1^{\prime}$};
\draw (1+10.25,.75) -- (1+11,.75);
{[line width=1pt]
\draw[->] (1+10.5,.9) -- (1+10.5,1.35);
}
\node at (1+10.85,1.125) {$R_1$};
\draw (1+10.25,1.5) -- (1+11,1.5);
{[line width=1.5pt]
{[black]
\draw[loosely dotted] (1+10.5,1.75) -- (1+10.5,2.75);
}
}
\draw (1+10.25,3) -- (1+11,3);
\draw (1+10.25,3.75) -- (1+11,3.75);
\draw (1+10.25,4.5) -- (1+11,4.5);
}
}
\shade [ball color=blue] (1+9.75,5) circle (.16) node[](residual){};
\draw[->] (1+10.5,5)--(1+10.1,5) node[blue](myarrow1){};
\node[blue] at (1+10.75,4.9) {$Y_1^{\prime}$};
\node[black,left=of residual]{$Y_2$};
\end{tikzpicture}

\section{Towerplex Conditions}\label{conditions}
Below is a list of explicit conditions that can be used to prove theorem 
\ref{towerplex}. 
\begin{enumerate}
\item $\lim_{n\to \infty}r_n= 0$;
\item $\sum_{n=1}^{\infty}r_n=\sum_{n=1}^{\infty}s_n=\infty$; 
\item $\lim_{n\to \infty}\mu (Y_n)=0$;
\item ${\epsilon_{n}}/{\max{\{M_n^1,M_n^{2}\}}} < \epsilon_{n+1}$;
\item $h_{n-1} < M_n^{1}, M_n^2 < h_n$;
\item $h_n\mbox{ sufficiently large such that equation \ref{SwitchingSetInd} holds}$;
\item $h_{n+1}\epsilon_n w_n > H_{m_n}+1$;
\item $\epsilon_{n+1} (H_{m_n}+1) < \epsilon_n w_n$;
\item $H_{m_{n+1}}\geq H_{m_n}$.
\end{enumerate}
If $r_n={1}/{2(n+2)}$ and $s_n=1/2$, and $\epsilon_n$ is sufficiently 
small such that Lemma \ref{CanParLem} holds, then we 
have a canonical towerplex construction.

\end{document}